\documentclass[11pt]{article}

\usepackage{amsmath,amssymb,geometry,url,graphicx,xcolor}
\usepackage[english]{babel}
\usepackage{amsxtra,amssymb,amsthm,hyperref}
\usepackage{todonotes}
\usepackage{float}
\usepackage{comment}
\usepackage[T1]{fontenc}
\usepackage{lmodern,microtype,tikz}
\newcommand{\rev}[1]{{\color{orange}#1}}

\hypersetup{hidelinks}
\DeclareMathOperator{\Forb}{Hom-Forb}
\DeclareMathOperator{\Hom}{Hom}

\DeclareMathOperator{\ar}{ar}
\DeclareMathOperator{\Th}{Th}
\newcommand{\cF}{\mathcal{F}}

\newcommand{\bF}{\mathfrak{F}}
\newcommand{\bA}{\mathfrak{A}}
\newcommand{\bB}{\mathfrak{B}}
\newcommand{\CSP}{\mathrm{CSP}}
\newcommand{\supp}{\mathrm{supp}}
\newcommand{\red}[1]{{\color{red}{#1}}}

\newcommand{\ignore}[1]{}
\DeclareMathOperator{\Csp}{CSP}
\DeclareMathOperator{\Col}{Col}
\DeclareMathOperator{\Aut}{Aut}

\newcommand{\cG}{\mathcal{G}}

\newtheorem{theorem}{Theorem}
\newtheorem{proposition}[theorem]{Proposition}
\newtheorem{corollary}[theorem]{Corollary}

\newtheorem{question}[theorem]{Question}
\newtheorem{lemma}[theorem]{Lemma}
\newtheorem{remark}[theorem]{Remark}
\newtheorem{definition}[theorem]{Definition}
\newtheorem{example}[theorem]{Example}

\title{
Asymptotic Theories of Classes  Defined by \\
Forbidden Homomorphisms}

\author{Manuel Bodirsky, Colin Jahel\footnote{Institut f\"ur Algebra, TU Dresden. Email: $\{$manuel.bodirsky,colin.jahel$\}$@tu-dresden.de. Both authors received funding from the DFG (Project FinHom, Grant 467967530) from the ERC (Grant Agreement no. 101071674, POCOCOP). Views and opinions expressed are however
those of the authors only and do not necessarily reflect those of the European Union or the European Research
Council Executive Agency.}
  }

\begin{document}

\maketitle
\thispagestyle{empty}

\begin{abstract}
We study first-order almost-sure theories of classes of finite graphs and
digraphs defined by homomorphically forbidding a set $\mathcal F$ of finite
structures. For undirected graphs, a complete description follows from the
Kolaitis--Pr\"omel--Rothschild theorem. For classes $\Csp(D)$ with a finite
oriented template $D$, we show that the colour-class distributions
of random $D$-coloured digraphs concentrate near distributions supported
on complete multipartite subgraphs that maximize density and then
entropy. We give sufficient conditions for $\Csp(D)$ to have an
$\omega$-categorical almost-sure theory, and sufficient conditions for its
almost-sure theory to be incomplete. Finally, we construct a core oriented
graph $D$ with $25$ vertices for which $\Csp(D)$ fails first-order
convergence: the probability of a fixed primitive positive sentence tends
to different limits along even and odd orders.
\end{abstract}


\section{Introduction}
A famous result of Glebskij, Kogan, Ligons'kii, and  Talanov~\cite{Gleb69} and
Fagin~\cite{Fagin} states that
the almost-sure theory $T_{\mathcal G}$ of the class $\mathcal G$ of all 
finite undirected graphs is \emph{complete}, i.e., 
every first-order property of graphs holds asymptotically almost surely, or its negation holds asymptotically almost surely (this is sometimes referred to as a \emph{first-order 0-1 law}). Moreover, $T_{\mathcal G}$ 
 is \emph{$\omega$-categorical}, i.e., it has exactly one countable model up to isomorphism, the so-called \emph{Rado graph}~\cite{Rado}, which is of fundamental importance in model theory and can be described as the Fra{\"i}ss\'e-limit of the class of all finite graphs~\cite{HodgesLong}. 
 
If instead of ${\mathcal G}$ we consider the class ${\mathcal G}_3$ of all finite \emph{triangle-free} graphs,
i.e., finite graphs that do not contain a $K_3$ subgraph ($K_k$ denotes the clique with $k$ vertices), we still get a first-order 0-1 law, and the resulting
almost-sure theory again has a single countably infinite model $M_2$. However, somewhat surprisingly, this  model is \emph{not} the Fra{\"i}ss\'e-limit of the class of all finite undirected triangle-free graphs: Erd\H{o}s, Kleitman, and Rothschild~\cite{ErdoesKleitmanRothschild} proved that if $n$ tends to infinity, almost all triangle-free graphs with $n$ vertices are \emph{bipartite}, i.e., have a homomorphism to $K_2$. Clearly, 
there are triangle-free graphs $G$ without a homomorphism to $K_2$ (e.g., the cycle with five vertices, $C_5$). The structure $M_2$ does not contain copies of $G$, and hence $M_2$ is not the Fra{\"i}ss\'e-limit of the class of all finite triangle-free graphs. This phenomenon has been studied systematically in~\cite{Koponen12}.

The mentioned result 
about triangle-free graphs has been generalized by Kolaitis, Pr\"omel, and Rothschild~\cite{KolaitisProemel} to the class of finite graphs without a copy of $K_k$, for every $k \geq 3$.

\begin{theorem}[Kolaitis, Pr\"omel, and Rothschild~\cite{KolaitisProemel}]\label{thm:KPR}
For every $k \geq 3$, a random $K_k$-free graph is asymptotically almost surely in $\Csp(K_{k-1}) := \{G \text{ finite} \mid G \to K_{k-1} \}$.
\end{theorem}

In fact, Kolaitis, Pr\"omel, and Rothschild determined the almost-sure theory of the class of $K_k$-free graphs (see Theorem~\ref{thm:KPR-2} below). 
In subsequent work, Pr\"omel and Steger~\cite{ProemelStegerCritical} studied classes of finite graphs that are given by forbidding some graph $H$ as a \emph{subgraph}, not necessarily induced.
In the present article, we study graph classes that are specified by forbidding a set $\cF$ of finite graphs \emph{homomorphically}, rather than forbidding them as (induced or weak) subgraphs. In particular, we will be interested in the almost-sure theory of the class 
$$\Forb(\cF) := \{ G \in {\mathcal G} \mid \text{no } F \in {\mathcal F} \text{ has a homomorphism to } G \}.$$
Note that that if $\cF = \{K_n\}$ then there is no difference between forbidding $\cF$ as subgraphs and forbidding $\cF$ homomorphically. 
 
Quite remarkably, one can obtain a complete description of the almost-sure theory of $\Forb(\cF)$ 
  for every set $\cF$ of finite graphs
The non-empty families
are covered by Theorem~\ref{thm:graphs-class}; the empty family is
handled in the next paragraph.

Let $\chi(G)$ be the \emph{chromatic number} of $G$, i.e., the smallest $n$ such that $G$ has a homomorphism to $K_n$. If $\cF = \emptyset$ then $\Forb(\cF) = \cG$ and the almost-sure theory is the theory of the Rado graph~\cite{Rado}. 
For non-empty $\cF$ containing no edgeless graph, define
$k_\cF := \min_{G \in {\cF}} \chi(G)$. 
We show in Section~\ref{sect:undir}  that 
the graphs in 
$\Forb(\cF)$ are asymptotically almost surely 
$(k_{\cF}-1)$-partite, that 
$\Forb(\cF)$ satisfies a first-order 0-1 law, and that
 the almost-sure theory of $\Forb(\cF)$ is $\omega$-categorical, i.e., has an up to isomorphism unique countable model.

 If $\cF$ is a set of finite \emph{directed} graphs (short \emph{digraphs}), then much less is known about the almost-sure theory of $\Forb(\cF)$, even if $\cF$ is finite. Some results have been obtained by K\"uhn, Osthus, Townsend, and Zhao~\cite{KOTZ17} for 
 $\Forb(T_3)$ and $\Forb(\vec C_3)$, 
 where $T_3$ is the transitive tournament with three vertices and $\vec C_3$ is the directed cycle with three vertices, answering questions that have been raised in model theory~\cite{Cherlin}. 
 In particular, they proved that asymptotically almost surely the
digraphs in $\Forb(T_3)$ have bipartite underlying undirected
graphs, meaning that their vertices can be partitioned into two
sets with no arcs inside either set.
 It is easy to come up with infinite sets ${\mathcal F}$ of finite digraphs such that $\Forb({\mathcal F})$ does not have a 0-1 law (see Example~\ref{expl:no-01}). 

To formulate our results, we use the following standard terminology.
Two digraphs $G$ and $H$ are called \emph{homomorphically equivalent} if $G \to H$ and $H \to G$. A digraph $G$ is called a \emph{core} if it is not homomorphically equivalent to one of its proper subgraphs. 
We write $\Csp(H)$
   (for \emph{constraint satisfaction problem})  
for the class of all finite digraphs $G$ with a homomorphism to $H$;
the digraph $H$ is called the \emph{template} of the CSP.  

If $\cF$ is a set of digraphs, then we write $\Forb(\cF)$ for the class of finite digraphs $G$ such that no digraph in $\cF$ maps homomorphically to $G$. 
It is easy to see that for every digraph $H$ there exists a (potentially infinite) set of finite digraphs ${\mathcal F}$ such that $\Csp(H) = \Forb({\mathcal F})$. A digraph is \emph{oriented} if it has neither loops nor pairs of opposite arcs. 

\subsection{Our Results}
For a class $\mathcal C$, let $\varphi_n^{\mathcal C}$ denote the probability
that a uniformly random member of $\mathcal C$ on $\{1,\ldots,n\}$ satisfies
the first-order sentence $\varphi$. The class has \emph{first-order
convergence} if this probability converges for every $\varphi$; its
\emph{almost-sure theory} consists of the sentences whose probabilities
tend to one. We write $K_{k_1,\ldots,k_\ell}$ for the complete
$\ell$-partite undirected graph with the indicated part sizes.

\begin{theorem}\label{thm:main-csp}
Let $D$ be a finite non-empty oriented graph, and let $S(D)$ be the set
of its induced subgraphs that are orientations of some
$K_{k_1,\ldots,k_\ell}$ and for which
$(\ell,\prod_{i=1}^\ell k_i)$ is lexicographically maximal.
That is, first maximize the number $\ell$ of non-empty parts,
and then maximize the product of their sizes.
If $S(D)$ contains at least one core and all its cores are
isomorphic, then $\Csp(D)$ has a complete, $\omega$-categorical
almost-sure theory.
If $S(D)$ contains two homomorphically inequivalent cores, then the
almost-sure theory of $\Csp(D)$ is incomplete.
\end{theorem}

\begin{remark}
The case where $S(D)$ contains no core is not covered by these
conditions. 
\end{remark}

\begin{theorem}\label{thm:non-convergence}
There exists a core oriented graph $D$ with $25$ vertices such that
$\Csp(D)$ does not have first-order convergence.
\end{theorem}

The example has two optimal subgraphs of equal density and entropy (see Definitions \ref{def:density} and Section \ref{sect:entropy}).
Their relative frequencies have different limits along even and odd
orders. Section~\ref{sect:non-convergence} constructs these
subgraphs $C_1,C_2$; the conjunctive query of $C_1$ is the sentence
with different limits.

\subsection{Connection with Model Theory}
 The $\omega$-categorical theories 
 that arise in our main result, Theorem~\ref{thm:main-csp}, 
 satisfy a property that is of interest in model theory: they are \emph{pseudo finite}, i.e., every finite subset of the theory has a finite model. An example of a pseudo-finite theory is the theory of the Rado graph, which follows from the fact that its first-order theory equals the almost-sure theory of finite undirected graphs. A class of  theories whose pseudo-finiteness can be shown similarly was isolated by Kruckman~\cite{n-disjoint}. Showing that a theory is pseudo finite can be quite challenging.  
 It is a major open question whether the Fra{\"i}ss\'e-limit of the class of all finite triangle-free graphs has a pseudo-finite theory (see, e.g.,~\cite{EvenZoharLinial}). 
 The mentioned result of Erd\H{o}s, Kleitman, and Rothschild shows that this problem cannot be approached via almost sure theories of finite structures in the same way as for the Rado graph. More generally, 
 Cherlin~\cite[Page 16]{Cherlin} asked for which amalgamation classes of finite structures described by forbidding 
finitely many substructures have a 0-1 law, and if so, what the limit theory looks like. 
 
\subsection{Connection with Database Theory}
Note that sets ${\mathcal F}$ of finite structures that are homomorphically forbidden  can be interpreted from a database perspective: 
each structure $\bF \in {\mathcal F}$ can be viewed as a so-called \emph{conjunctive query} $\varphi$, by turning the elements of the structure into existentially quantified variables and the relations of the structure into a conjunction of atomic formulas. 
Conjunctive queries are the most important class of queries encountered in database theory. 
It is well-known and easy to see that a finite structure $\bB$
(i.e., a \emph{database})
satisfies 
$\varphi$ if and only if $\bF$ has a homomorphism into $\bB$; see~\cite{ChandraMerlin}. 
Results about $\Forb({\mathcal F})$ can then be interpreted as 
statements about the uniform distribution on finite structures that
do \emph{not} satisfy 
the given set of conjunctive queries. 
  
\section{Preliminaries}
A \emph{digraph} $D=(V,E)$ consists of a set of vertices $V=V(D)$ and a set of edges $E = E(D) \subseteq V(D)^2$; an \emph{oriented
graph} has no loops and no pair of opposite arcs.
The \emph{order} of a graph or digraph $G$ is its number
$|V(G)|$ of vertices.
A subgraph is \emph{induced} when it contains exactly the
edges inherited on its vertex set. An \emph{embedding} is an isomorphism
onto an induced subgraph. We write $D-v$ for the induced subgraph
obtained by deleting $v$, and $D_1\uplus D_2$ for the disjoint union.
A digraph is \emph{connected} if any two vertices can be joined
by a path when arc directions are ignored; its \emph{components}
are the maximal connected induced subgraphs.
We write $T_k$ for the transitive tournament on $k$ vertices and
$\vec C_k$ for the directed cycle of length $k$.

\subsection{Undirected graphs}\label{sect:undirected-prelim}
An \emph{undirected graph} $H$ has a vertex set $V(H)$ and an edge set
$E(H)\subseteq\binom{V(H)}2$; in particular, it has neither loops
nor multiple edges. As a relational structure, it carries the
corresponding symmetric, irreflexive adjacency relation. When
summing over $E(H)$, we count each unordered edge once.

We write $K_k$ for the complete graph on $k$ vertices.
A set of vertices is \emph{independent} if no two of its vertices
are adjacent. A graph is \emph{$k$-partite} if its vertex set is
the disjoint union of $k$ independent sets, some of which may be
empty; equivalently, it admits a homomorphism to $K_k$.
For a finite non-empty graph $H$, its \emph{chromatic number}
$\chi(H)$ is the least such $k$, and its \emph{clique number}
$\omega(H)$ is the largest $k$ for which $H$ contains a copy of
$K_k$. A $2$-partite graph is called \emph{bipartite}.

For positive integers $k_1,\ldots,k_\ell$, the \emph{complete multipartite
graph} $K_{k_1,\ldots,k_\ell}$ has independent parts of these sizes,
with every pair of vertices from distinct parts adjacent.
An \emph{orientation} of an undirected graph replaces each edge
$\{u,v\}$ by exactly one of the arcs $(u,v)$ and $(v,u)$.
Conversely, the \emph{underlying undirected graph} of an oriented
graph is obtained by forgetting the directions of its arcs.
An orientation of a complete graph is a \emph{tournament}.

\subsection{Relational structures}
Although we are ultimately interested in digraphs, it is convenient
to work in the more general setting of relational structures.
Our arguments also use digraphs with distinguished vertex colour
classes, and this setting lets us treat both coloured and uncoloured
digraphs uniformly. In particular, we can formulate the model-theoretic
and almost-sure notions once for all these structures.

We use standard terminology from model theory~\cite{HodgesLong}.
A \emph{relational signature} is a set of relation symbols $R$, each
equipped with an \emph{arity} $\ar(R)\in\mathbb N_{\geq1}$.
For a relational signature $\tau$, a \emph{$\tau$-structure}
$\bA=(A;(R^{\bA})_{R\in\tau})$ consists of a set $A$, called its
\emph{domain}, and, for each $R\in\tau$, a relation
$R^{\bA}\subseteq A^{\ar(R)}$, called the \emph{interpretation} of $R$.
The structure is \emph{finite} if its domain is finite.
In particular, a digraph $D$ is exactly a structure for the signature
$\tau=\{E\}$ with $\ar(E)=2$: its domain is $V(D)$ and the
interpretation $E^D$ is its arc relation $E(D)$. Thus, the definitions
below specialize to the usual notions for digraphs.

For two $\tau$-structures $\bA$ and $\bB$ with domains $A$ and $B$,
a \emph{homomorphism} from $\bA$ to $\bB$ is a map $h\colon A\to B$
that preserves every relation: for each $R\in\tau$ of arity $k$
and every $(a_1,\ldots,a_k)\in A^k$,
\[
 (a_1,\ldots,a_k)\in R^{\bA}
 \quad\Longrightarrow\quad
 (h(a_1),\ldots,h(a_k))\in R^{\bB}.
\]
We write $\bA\to\bB$ if such a map exists, $\bA\not\to\bB$
otherwise, and $\Hom(\bA,\bB)$ for the set of all such maps.
An \emph{endomorphism} is a homomorphism from a structure to itself.
An \emph{induced substructure} on $B\subseteq A$ interprets each
$R\in\tau$ as $R^{\bA}\cap B^{\ar(R)}$. An \emph{isomorphism}
from $\bA$ to $\bB$ is a bijective homomorphism whose inverse
is a homomorphism from $\bB$ to $\bA$; an \emph{automorphism} is
an isomorphism from a structure to itself. An \emph{embedding}
is an isomorphism onto an induced substructure.
We write $\Aut(\bA)$ for the \emph{automorphism group} of $\bA$.
For a finite structure, every bijective endomorphism is an
automorphism.

Recall that we are interested in classes of the form $\Csp(H)$, i.e., for a given structure $H$, the class of structures admitting a homomorphism to $H$. In this setup, we call $H$ a \emph{template}. Moreover, for a family of structures $\mathcal{F}$, we write $\Forb(\mathcal{F})$ for the class of structure that do not admit a homomorphism from any element of $\mathcal{F}$.

An \emph{expansion} adds interpretations of new relation symbols
without changing the domain or the existing relations; a
\emph{reduct} forgets some of these interpretations.
For example, vertex colours can be recorded by adding a unary
relation symbol for each colour, interpreted as its colour class.
Forgetting these unary relations recovers the underlying digraph.
By contrast, an \emph{extension} of a $\tau$-structure $\bA$
is a $\tau$-structure $\bB$ containing $\bA$ as an induced
substructure; we write $\bA\subseteq\bB$. Thus, an extension
may add elements to the domain, but keeps the signature unchanged
and agrees with all relations on the original domain.
The \emph{first-order theory} $\Th(M)$ of $M$ is the set of
first-order sentences true in $M$.
A structure is \emph{homogeneous} if every isomorphism between its
finite induced substructures extends to an automorphism.
Its \emph{age} is the class of its finite induced substructures up
to isomorphism. A theory is \emph{$\omega$-categorical} if it has,
up to isomorphism, exactly one countable model. If $M$ is a countable homogeneous structure in a finite
relational signature, then $\Th(M)$ is $\omega$-categorical~\cite[Theorem~7.4.1]{HodgesLong}.
The theory of every reduct of $M$ is also $\omega$-categorical.

We also view an undirected graph as a relational structure with
one symmetric, irreflexive binary relation $E$: each edge $\{u,v\}$
is represented by both ordered pairs $(u,v)$ and $(v,u)$.
Thus, undirected graphs are represented by symmetric digraphs.
For an undirected template $H$, however, we always restrict the inputs:
$\Csp(H)$ consists of finite undirected graphs admitting a homomorphism
to $H$, and $\Col(H)$ consists of pairs $(G,f)$ with $G$ such an
undirected graph and $f\in\Hom(G,H)$. In particular, the relation
$E$ must be symmetric in every input. Allowing arbitrary digraphs
mapping to the same symmetric template would give different classes
and different counts: for an allowed pair of distinct vertices, an
undirected edge is either present or absent, whereas its two arcs
could be chosen independently in the digraph setting.

\subsection{Almost-sure theories}
Let $\tau$ be a finite relational signature and $\mathcal C$
an isomorphism-closed class of finite $\tau$-structures. Write
$\mathcal C_n$ for its members with domain $[n]=\{1,\ldots,n\}$.
Throughout, the classes
under consideration have $\mathcal C_n\ne\varnothing$ for all
sufficiently large $n$. For a first-order $\tau$-sentence
$\varphi$, define its probability under the uniform distribution
on $\mathcal C_n$ by
\[
 \varphi_n^{\mathcal C}
 :=\frac{|\{\bA\in\mathcal C_n \mid \bA\models\varphi\}|}{|\mathcal C_n|}.
\]
When the limit $\lim_{n \to \infty} \phi_n^{\mathcal C}$ exists, denote it by $\varphi^{\mathcal C}$.
The class $\mathcal C$ has \emph{first-order convergence} if
this limit exists for every first-order $\tau$-sentence. These
definitions apply, in particular, to graphs, digraphs, and their
expansions by colour predicates.
The \emph{almost-sure theory} is
$T_{\mathcal C}:=\{\varphi \mid \varphi^{\mathcal C} = 1\}$.
The class has a \emph{first-order $0$--$1$ law} if this theory is
\emph{complete}, i.e., every sentence or its
negation belongs to $T_{\mathcal C}$; equivalently, every
$\varphi_n^{\mathcal C}$ converges to either $0$ or $1$.
This theory is \emph{consistent}: any finite subset holds
simultaneously in some member of $\mathcal C$.
We abbreviate \emph{asymptotically almost surely} by ``a.a.s.'';
it means that the probability of the stated property tends to one
as $n\to\infty$. 

\begin{lemma}\label{lem:transfer}
If $\mathcal C'\subseteq\mathcal C$ and
$|\mathcal C'_n|/|\mathcal C_n|\to1$, then, for every first-order $\tau$-sentence $\varphi$,
\[
 |\varphi_n^{\mathcal C}-\varphi_n^{\mathcal C'}|
 \le 1-\frac{|\mathcal C'_n|}{|\mathcal C_n|}\longrightarrow0.
\]
\end{lemma}
\begin{proof}
Condition the uniform distribution on $\mathcal C_n$ on membership
in $\mathcal C'_n$. The resulting distribution is uniform on
$\mathcal C'_n$, and the complementary event has the displayed probability.
\end{proof}
In particular, the limit
$\varphi_n^{\mathcal C}$ exists if and only if
the limit $\varphi_n^{\mathcal C'}$ exists, 
and if these limits exist, then they agree, and the two
classes have the same almost-sure theory.

\subsection{Cores and conjunctive queries}\label{sect:graph-hom}
The goal of this short subsection is to introduce notions from CSP theory, cores that aim at simplyfying templates, and conjunctive
queries, that characterize the existence of homomorphisms.

A finite digraph is a \emph{core} if every endomorphism is an
automorphism. Every finite non-empty digraph is homomorphically
equivalent to a core, unique up to isomorphism~\cite{HNBook}.
The same conventions apply to undirected graphs.
Two digraphs are \emph{homomorphically incomparable} if there
is no homomorphism in either direction.
For an undirected template, $\Csp(H)$ denotes a class of undirected
graphs; treating the same template as a symmetric digraph would give
a different class of inputs.

The \emph{conjunctive query} of a finite digraph $H$ is the sentence
\[
 \theta_H:=\exists(x_v)_{v\in V(H)}
       \bigwedge_{(u,v)\in E(H)}E(x_u,x_v).
\]
Thus, $G\models\theta_H$ if and only if $H\to G$. This is a
\emph{primitive positive} sentence: it uses only existential
quantifiers and a conjunction of atomic formulas.

\subsection{Generic coloured structures}
Our approach to random digraphs in $\Csp(D)$ is to first study a different class of structures, where we retain a particular homomorphism to $D$ as
part of the structure, recording its fibres as colour classes.
From this setting we later derive our results about $\Csp(D)$. 
For each finite non-empty $D$, we construct the associated countable
homogeneous coloured structure $M^{\mathrm{col}}(D)$ and describe its theory by extension
axioms. These structures and their uncoloured reducts will provide
models for the almost-sure theories identified later in the paper.

\begin{definition}\label{def:D-col}
A \emph{$D$-coloured digraph} is a pair $(G,f)$ with $f\in\Hom(G,D)$.
We regard it as the relational structure $G_f$ obtained from $G$ by
adding the unary predicates $P_u=f^{-1}(u)$, for $u\in V(D)$.
Write $\Col(D)_n$ for all such structures with domain $[n]$.
Write $\Col(D)$ for the class of all finite $D$-coloured
digraphs. A \emph{colour-respecting map} preserves each predicate $P_u$;
in particular, a \emph{colour-respecting copy of $D$ in $(G,f)$} is an
induced copy with its vertex indexed by $u$ lying in $f^{-1}(u)$.
\end{definition}

The following lemma is well known and
a standard application of Fra\"iss\'e's theorem~\cite{HodgesLong}. 
It is a special case of 
Pech and Pech's existence and uniqueness theorem for universal
homogeneous coloured structures~\cite[Theorem~4.1]{PechPechConstraints},
applied to the class of all finite digraphs (or all finite undirected
graphs). 

\begin{lemma}\label{lem:AP}
For every finite non-empty digraph $D$ there is a countably infinite
homogeneous $D$-coloured digraph $M^{\mathrm{col}}(D)$ whose age consists of all
finite $D$-coloured digraphs. It is unique up to a colour-preserving
isomorphism.
\end{lemma}

We call $M^{\mathrm{col}}(D)$ the \emph{generic $D$-coloured digraph and} write $M(D)$ for the digraph obtained from $M^{\mathrm{col}}(D)$ by forgetting the colour
predicates and use the analogous terminology for undirected graphs.

The generic $D$-colored digraph $M^{\mathrm{col}}(D)$ has the \emph{extension
property}: if $B$ is a finite $D$-coloured digraph, then 
every colour-preserving embedding of a substructure $A$ of $B$ into $M^{\mathrm{col}}(D)$ extends
to an embedding of $B$. Indeed, embed $B$ using the age condition,
then use homogeneity to move the image of $A$ to the prescribed
copy. A \emph{one-point extension} has $B=A\cup\{b\}$ with
$b\notin A$; iterating these extensions gives the same property
for arbitrary finite $B$.


The \emph{universal axioms for $D$-coloured
structures} say that the predicates $P_u$ partition the domain and
that $E(x,y)$ is forbidden whenever the colours of $x,y$ do not
form an arc of $D$. 
For undirected graphs we also have universal sentences for expressing the symmetry and irreflexivity of the edge relation.

For a finite $D$-coloured digraph $A$, let
$\chi_A$ be its \emph{complete quantifier-free diagram}: the
conjunction specifying all arcs and non-arcs, all colour predicates
and their negations, and inequalities between distinct enumerated
elements. 
If $B$ is a $D$-coloured digraph, $b \in B$,
and $A$ is the substructure of $B$ with domain $B \setminus \{b\}$, then 
$M^{\mathrm{col}}(D)$ satisfies the following
\emph{$D$-colored one-point 
extension axiom} 
\[
 \forall\bar x\bigl(\chi_A(\bar x)\Rightarrow
                     \exists y.\,\chi_B(\bar x,y)\bigr).
\]
Here $A$ may be empty.
Also the following should be considered to be well-known, and we add a proof for the convenience of the reader.

\begin{lemma}\label{lem:ext-ax}
The theory of the structure $M^{\mathrm{col}}(D)$ from Lemma~\ref{lem:AP}
is axiomatized by the universal axioms for $D$-coloured structures
and $D$-colored extension axioms.
\end{lemma}
\begin{proof}
The generic $D$-coloured digraph satisfies the universal axioms by definition
and the extension axioms by the extension property just established.
Conversely, let $N,N'$ be countable models of these axioms. Both
are infinite: starting with the empty structure, successive
one-point extension axioms embed arbitrarily large edgeless
structures of any fixed colour.
One can then show that $N$ and $N'$ are isomorphic, using a standard
back-and-forth construction~\cite[Section~3.2]{HodgesLong}.
Thus, the axioms have, up to isomorphism, just one countable model,
namely $M^{\mathrm{col}}(D)$. If they did not decide some sentence, there
would be models of the axioms satisfying that sentence and its
negation. By the downward L\"owenheim--Skolem theorem~\cite[Section~3.1]{HodgesLong}, each would
have a countably infinite elementary substructure, contradicting
uniqueness. Hence, the axioms form the complete theory claimed.
\end{proof}

\begin{theorem}[Kolaitis, Pr\"omel, and Rothschild~\cite{KolaitisProemel}]
\label{thm:KPR-2}
For $k\ge2$, the almost-sure theory of the class of finite undirected
$k$-partite graphs is $\Th(M(K_k))$.
\end{theorem}
We call $M(K_k)$ the \emph{generic $k$-partite graph}. For $k=1$,
the corresponding assertion is immediate: $M(K_1)$ is a countably
infinite edgeless graph.

\subsection{Asymptotic Notation}
Unless stated otherwise, all asymptotic statements refer to
$n\to\infty$, with all other parameters fixed.
Let $f\colon\mathbb N\to\mathbb R$ be positive for all sufficiently
large $n$. We write $O(f)$ for the class of functions
$g\colon\mathbb N\to\mathbb R$ for which there exists a constant
$C>0$ such that for all sufficiently large $n \in {\mathbb N}$
\[
 |g(n)|\le Cf(n).
\]
We write $o(f)$ for the class of functions $g$ such that
$g(n)/f(n)\to0$, and $\Theta(f)$ for the class of functions $g$
for which there exist constants $c,C>0$ such that
\[
 cf(n)\le g(n)\le Cf(n)
\]
for all sufficiently large $n$.
In particular, $O(1)$ is the class of bounded functions, and
$o(1)$ is the class of functions tending to zero.

Sums and products of classes of functions are interpreted
pointwise, with a single function identified with its singleton
class. For example,
\[
 O(f)+g=\{h+g\mid h\in O(f)\}.
\]
Thus, $h\in g+o(f)$ means that $(h(n)-g(n))/f(n)\to0$.

We write $g \approx f$ if $g(n)/f(n)\to1$, or equivalently,
if $g\in f+o(f)$.
We use the same notation for sequences; for example,
$a_n\in\Theta(b_n)$ refers to the functions
$n\mapsto a_n$ and $n\mapsto b_n$.


\section{Forbidden Undirected Graphs}\label{sect:undir}
The following consequence of Theorem~\ref{thm:KPR} applies even to infinite forbidden
families $\mathcal F$.

\begin{theorem}\label{thm:graphs-class}
Let $\mathcal F$ be a non-empty set of finite undirected graphs, each
containing an edge, and put $k_{\mathcal F}=\min_{F\in\mathcal F}\chi(F)$.
Then $\Forb(\mathcal F)$ asymptotically almost surely consists of
$(k_{\mathcal F}-1)$-partite graphs, and its almost-sure theory is
$\Th(M(K_{k_{\mathcal F}-1}))$.
\end{theorem}

\begin{proof}
We have
\[
 \Csp(K_{k_{\mathcal F}-1})\subseteq\Forb(\mathcal F)
 \subseteq\Forb(K_{k_{\mathcal F}}).
\]
The first inclusion follows because no member of $\mathcal F$ is
$(k_{\mathcal F}-1)$-colourable. For the second, choose $F\in\mathcal F$
with $\chi(F)=k_{\mathcal F}$: if $K_{k_{\mathcal F}}\to G$, then
$F\to K_{k_{\mathcal F}}\to G$.
For $k_{\mathcal F}\ge3$, Theorem~\ref{thm:KPR} says that the ratio of the
labelled counts of the outer classes tends to one; for $k_{\mathcal F}=2$ both
outer classes consist of edgeless graphs. The claims follow from
Lemma~\ref{lem:transfer} and Theorem~\ref{thm:KPR-2}.
\end{proof}

\section{Forbidden Digraphs}\label{sect:digraphs}
In contrast to the undirected case, a class $\Csp(D)$ for a finite
oriented template need not have a first-order $0$--$1$ law.
We use $x\longrightarrow y$ as an abbreviation for $E(x,y)$ and
$x_1\longrightarrow x_2\longrightarrow x_3$ for $E(x_1,x_2)\wedge E(x_2,x_3)$.

\begin{example}\label{expl:no-01}
Let ${\mathcal F}$ be the set of all 
finite connected digraphs that map homomorphically neither to $\vec C_3$ nor to $T_3$.  
Then $\mathcal C := \Forb({\mathcal F}) = \Csp(D)$ where 
$D$ is the disjoint union of
$\vec C_3$ and $T_3$. 
If $\varphi$ is the primitive positive sentence 
$$ \exists x_1,x_2,x_3 
(
x_1 \longrightarrow x_2 \longrightarrow x_3 \wedge x_1 \longrightarrow x_3 
) $$
then 
$\varphi^{\mathcal C}$ is neither $0$ nor $1$, as we will see in Example~\ref{expl:combine}. \end{example}

In the following, we study classes of the form $\Forb(\cF) = \Csp(D)$ for finite digraphs $D$. 
Our approach to determine the asymptotic first-order theory of $\Csp(D)$ is to first determine the asymptotic theory of the class of $D$-coloured digraphs. In Section~\ref{sect:forget} and~\ref{sect:results} we then discuss how to transfer the results back to the original setting.

\subsection{Random \texorpdfstring{$D$}{D}-coloured digraphs}\label{sect:main}
Fix a finite non-empty oriented graph $D$, and write $r=|V(D)|$.
We use the uniform distribution on $\Col(D)_n$, so every pair
$(G,f)$ with $G\in\Csp(D)_n$ and $f\in\Hom(G,D)$ is equally likely.
In particular,
\begin{equation}\label{eq:col-csp-hom}
 |{\Col}(D)_n|=\sum_{G\in\Csp(D)_n}|\Hom(G,D)|.
\end{equation}
If $V(D)=\{v_1,\ldots,v_r\}$, choosing the fibres of $f$ and then
choosing the permitted arcs independently gives
\begin{equation}\label{eq:col-count}
 |{\Col}(D)_n|=
{\sum_{\substack{s_1,\ldots,s_r\in\mathbb Z_{\ge0}\\s_1+\cdots+s_r=n}}}
 \binom n{s_1,\ldots,s_r}
 2^{\sum_{(v_i,v_j)\in E(D)}s_is_j}.
\end{equation}

\subsection{Density}\label{sect:density}
Let $\mathcal Z$ be the \emph{probability simplex} on $V(D)$, equipped with
the \emph{$\ell_1$ norm} $\|x\|_1={\sum_{v\in V(D)}}|x_v|$, i.e.,  $$\mathcal Z=\{x\in\mathbb R^{V(D)} \mid x_v\ge0\text{ for
all }{v\in V(D)},\ \sum_{v\in V(D)}x_v=1\}.$$
For $x \in {\mathcal Z}$, we write 
$\supp(x)=\{v\in V(D) \mid x_v>0\}$ for its \emph{support}.

\begin{definition}\label{def:density}
For $x\in\mathcal Z$, put
\begin{equation}\label{eq:max}
 d(x):=\sum_{(u,v)\in E(D)}x_ux_v.
\end{equation}
The \emph{density} of $D$ is $d:=\max_{x\in\mathcal Z}d(x)$.
A \emph{density maximizer} is an $x \in {\mathcal Z}$ such that $d(x)$ is maximal; the compactness of ${\mathcal Z}$ guarantees the
existence of the maximum.
\end{definition}

A \emph{blow-up} replaces each vertex of a graph by a non-empty
independent set, with exactly the adjacencies inherited from the
original vertices; for digraphs it also preserves the directions.
The blow-up is \emph{nontrivial} if at least one vertex is
replaced by an independent set with more than one vertex.
Note that an orientation of a blow-up of $K_\ell$ need not be a blow-up of a
tournament.

\begin{theorem}[Motzkin--Straus~{\cite[Theorems~1 and~2]{GraphMaxima}}]\label{thm:MS}
If $H$ is a finite non-empty undirected graph with $\omega(H)=\ell$,
and $\mathcal Z_H$ is the probability simplex on $V(H)$, then
\[
 \max_{x\in\mathcal Z_H}\sum_{\{u,v\}\in E(H)}x_ux_v
 =\frac{\ell-1}{2\ell}.
\]
The subgraph induced by the support of every maximizer is a complete
$\ell$-partite graph, and each of its parts has total mass $1/\ell$.
Conversely, any distribution with these properties is a maximizer.
\end{theorem}
The expression for the maximum value  in this theorem is Theorem~1 of~\cite{GraphMaxima}.
For the assertion about supports, apply
\cite[Theorem~2]{GraphMaxima} to the subgraph induced by
the support of a maximizing distribution. The restricted
distribution has strictly positive coordinates, as required
by that theorem.


\begin{theorem}\label{thm:MS-oriented}
Let $\ell$ be the largest order of a tournament contained in $D$.
Then $d=(\ell-1)/(2\ell)$. The subgraph induced by the support of
every density maximizer is an orientation of a complete
$\ell$-partite graph, with mass $1/\ell$ on each part. Conversely,
every distribution of this form is a density maximizer.
\end{theorem}
\begin{proof}
Apply Theorem~\ref{thm:MS} to the underlying undirected graph of $D$.
\end{proof}

A fixed
tournament can belong to several inclusion-wise maximal complete
multipartite subgraphs. Figure~\ref{fig:orientations} illustrates
this in the digraphs $G_L$ and $G_R$.

\begin{figure}[H]
\begin{center}
\begin{tikzpicture}[
    every node/.style={circle, fill=black, inner sep=2.4pt},
    every label/.style={fill=none, text=orange, font=\small},
    arrow/.style={->, line width=1.1pt, line cap=round},
  ]
  \node[label=left:$a$] (V0) at (-7,0) {}; 
  \node[label=right:$b$] (V1) at (-5,0) {}; 
  \node[label=right:$c$] (V2) at (-5,2) {}; 
  \node[label=left:$d$] (V3) at (-5,4) {}; 
  \node[label=right:$e$] (V4) at (-3,4) {}; 
  \draw[arrow, draw=orange] (V0) to (V3);
  \draw[arrow] (V0) to[bend right=25]  (V4); 
  \draw[arrow] (V1)                   to (V2);
  \draw[arrow] (V1) to[bend right=30] (V4);
  \draw[arrow] (V2)                   to (V3);
  \node[label=left:$a$] (A) at (0,0) {}; 
  \node[label=right:$b$] (B) at (2,0) {}; 
  \node[label=left:$c$] (C) at (2,2) {}; 
  \node[label=left:$d$] (D) at (2,4) {}; 
  \node[label=right:$e$] (E) at (4,4) {}; 
  \draw[arrow] (A)                   to (C);
  \draw[arrow] (A)                   to (D);
  \draw[arrow, draw=orange] (B) to (C);
  \draw[arrow, draw=orange] (C) to (D);
  \draw[arrow] (C)                   to (E);
  \draw[arrow] (B) to[bend right=30] (E);
  \draw[arrow, draw=orange] (B) to[bend left=55] (D); 
  \node[fill=none, text=orange] at (-5,-0.6) {$G_L$};
  \node[fill=none, text=orange] at (2,-0.6) {$G_R$};
\end{tikzpicture}
\end{center}
\caption{A fixed tournament can lie in several inclusion-wise
maximal complete multipartite subgraphs. In $G_L$, the highlighted
arc $a\to d$ lies in the maximal complete bipartite orientations on
$\{a,d,e\}$ and $\{a,c,d\}$. In $G_R$, the highlighted tournament on
$\{b,c,d\}$ lies in the maximal complete tripartite orientations on
$\{a,b,c,d\}$ and $\{b,c,d,e\}$.}

\label{fig:orientations}
\end{figure}

\subsection{Typical colour profiles and entropy}\label{sect:entropy}
When counting the number of $D$-colored digraphs in \eqref{eq:col-count}, each choice of colour-class sizes
contributes two factors: a multinomial coefficient counting
the colour assignments, and a power of $2$ counting the
compatible digraphs. The preceding subsection studied the
second factor through the density function. We now turn
to the multinomial coefficient.

For a colouring $f\colon[n]\to V(D)$, call the vector
$g=(g_v)_{v\in V(D)}$, where $g_v=|f^{-1}(v)|/n$, its
\emph{colour profile}. 
For a colour profile $x$, the number of colourings 
is $\binom n{nx}$. To estimate this number, we associate
with each $x\in\mathcal Z$ its \emph{entropy}
$$
 h(x):=-\sum_{v\in\supp(x)}x_v\log_2x_v.
$$
The following bounds relate multinomial coefficients to entropy. Throughout this section, let $r:=|V(D)|$.
\begin{lemma}[{\cite[Theorem~11.1.3]{CoverThomas}}]
\label{lem:multinomial-entropy}
Let $n$ be a positive integer, and let
$x\in\mathcal Z$ satisfy $nx_v\in\mathbb Z_{\ge0}$
for every $v\in V(D)$. Then
\begin{equation}\label{eq:multinomial-entropy-bound}
 \frac{2^{nh(x)}}{(n+1)^r}
 \le \binom n{nx}
 \le 2^{nh(x)}.
\end{equation}
\end{lemma}
Together with the density calculation, this will
allow us to identify the colour profiles that account for
almost all $D$-coloured digraphs.

The function $h$ is continuous on $\mathcal Z$, including at points
where the support changes, since $\lim_{t \searrow 0}t\log_2 t=0$. 
Let 
\begin{align*}
\mathcal P & :=\{x\in\mathcal Z \mid d(x)=d\} \\
\lambda & :=\max_{x\in\mathcal P}h(x) \\
\mathcal O & :=\{x\in\mathcal P \mid h(x)=\lambda\}.
\end{align*} 
Let $\ell$ be the largest order of a tournament contained in $D$.

\begin{lemma}\label{lem:entropy}
The supports of the elements of $\mathcal O$ are exactly the vertex
sets of the members of $S(D)$ from Theorem~\ref{thm:main-csp}.
For $Q\in S(D)$ with parts $A_1,\ldots,A_\ell$, the unique element
$\delta_Q\in\mathcal O$ with support $V(Q)$ is
\begin{equation}\label{eq:distr}
 \delta_Q(v)=\frac1{\ell|A_i|}\quad(v\in A_i,\ {1\le i\le\ell}),
 \qquad \delta_Q(v)=0\quad({v\in V(D)\setminus V(Q)}).
\end{equation}
In particular, we have
\begin{equation}\label{eq:entropy-value}
 \lambda=\log_2\ell+\frac1\ell\log_2 \left(\prod_{i=1}^{\ell}|A_i|\right).
\end{equation}
Every $Q\in S(D)$ is inclusion-wise maximal among the induced
orientations of complete $\ell$-partite subgraphs of $D$.
\end{lemma}

\begin{proof}
Among all induced subgraphs of $D$ that are orientations of complete
$\ell$-partite graphs with non-empty parts, let $P_*$ be the largest
product of the part sizes.
Since the number of parts is at most $\ell$, 
the set of induced subgraphs attaining this maximum is precisely $S(D)$.

By Theorem~\ref{thm:MS-oriented}, any $x\in\mathcal P$ has support
inducing an orientation $Q$ of a complete $\ell$-partite graph,
and its parts $A_1,\ldots,A_\ell$ satisfy
$\sum_{v\in A_i}x_v=1/\ell$. Put $k_i=|A_i|$. Strict concavity
of $t\mapsto-t\log_2t$ gives, for each $i\in\{1,\ldots,\ell\}$,
\[
 -\sum_{v\in A_i}x_v\log_2x_v
 \le -k_i\frac{1}{\ell k_i}\log_2\frac{1}{\ell k_i}
 =\frac1\ell\log_2(\ell k_i),
\]
with equality if and only if $x_v=1/(\ell k_i)$ for every
$v\in A_i$. 
Summing over the parts yields
\begin{align}  h(x)\le\log_2\ell+\frac1\ell\log_2\!\left(\prod_{i=1}^\ell k_i\right)
 \le\log_2\ell+\frac1\ell\log_2P_*.
 \label{eq:entropy}
\end{align}

Let $Q$ be any induced subgraph of $D$ that is an orientation
of a complete $\ell$-partite graph, with non-empty parts
$A_1,\ldots,A_\ell$. Put $k_i:=|A_i|$ for
$i\in\{1,\ldots,\ell\}$, and define $\delta_Q$ by
\eqref{eq:distr}. This vector 
 has mass $1/\ell$ on each part, so it belongs to
$\mathcal P$. Since its coordinates are constant on each part, equality holds
in the first inequality in~\eqref{eq:entropy}, i.e.,
$$
 h(\delta_Q)=\log_2\ell+
 \frac1\ell\log_2\!\left(\prod_{i=1}^{\ell}k_i\right).
$$
Choosing $Q$ with $\prod_{i=1}^{\ell}k_i=P_*$ also gives equality in the
second inequality in~\eqref{eq:entropy}. Hence,
$\lambda=\log_2\ell+\ell^{-1}\log_2P_*$.
This proves
\eqref{eq:entropy-value}.

If $x\in\mathcal O$, equality must hold in both bounds.
Consequently, its support induces a member $Q\in S(D)$, and
equality in each part forces $x=\delta_Q$. Conversely, if
$Q\in S(D)$, the preceding calculation shows
$h(\delta_Q)=\lambda$, hence $\delta_Q\in\mathcal O$. The
multipartition is determined, up to ordering, by the underlying
graph: two vertices belong to the same part exactly when they
are equal or non-adjacent. Thus, the vector with this support
is unique, proving both the support assertion and uniqueness.

Finally, if $Q\in S(D)$ were properly contained in an induced
orientation $Q'$ of a complete $\ell$-partite graph, then,
after reindexing its parts, we would have $A_i\subseteq A_i'$
for every $i$, with at least one inclusion strict. Hence,
$\prod_{i=1}^{\ell}|A_i'|>\prod_{i=1}^{\ell}|A_i|=P_*$,
a contradiction. This proves inclusion-wise maximality.
\end{proof}

We call a member $Q\in S(D)$ an \emph{optimal support},
identifying it with its vertex set when discussing distributions:
it supports a distribution maximizing density and then entropy.

For $U\subseteq V(D)$ we write
$g(U)=\sum_{v\in U}g_v$ for its \emph{total mass}.
When every $ng_v$ is integral, $\binom n{ng}$ denotes the
\emph{multinomial coefficient} $$\binom n{ng} := n!/\prod_{v\in V(D)}(ng_v)!$$
For $\varepsilon>0$ put
\[
 \mathcal O_\varepsilon
 =\{x\in\mathcal Z \mid \|x-\delta\|_1<\varepsilon
                    \text{ for some }\delta\in\mathcal O\}.
\]

\begin{theorem}\label{thm:main}
For every $\varepsilon>0$, a uniformly random $(G,f)\in\Col(D)_n$
asymptotically almost surely has some $Q\in S(D)$ such that
$\operatorname{im}f=V(Q)$ and 
\[
 \left\|\left(\frac{|f^{-1}(v)|}{n}\right)_{v\in V(D)}
                -\delta_Q\right\|_1<\varepsilon.
\]
\end{theorem}

To prove Theorem~\ref{thm:main}, we need the following three intermediate results.

\begin{lemma}\label{lem:lowebound}
For some constant $C$ and all sufficiently large $n$,
\[
 |{\Col}(D)_n|\ge 2^{dn^2+n\lambda-C\log_2 n}.
\]
\end{lemma}
\begin{proof} 
Choose \(Q\in S(D)\), with parts \(A_1,\ldots,A_\ell\). Round the numbers \(n\delta_Q(v)\) down, and increase sufficiently many coordinates indexed by \(V(Q)\) by one to obtain integers \(N_v\) satisfying
\[
\sum_{v\in V(D)}N_v=n,\qquad
N_v=0\quad(v\notin V(Q)),\qquad
|N_v-n\delta_Q(v)|\le1.
\]
Put
\[
N_i:=\sum_{v\in A_i}N_v=\frac n\ell+O(1).
\]
The number of colourings $f\colon[n]\to V(D)$ with
$|f^{-1}(v)|=N_v$ for every $v\in V(D)$ is
$\binom n{(N_v)_{v\in V(Q)}}$.
Put $x_v:=N_v/n$ for $v\in V(D)$.
By the lower bound in
Lemma~\ref{lem:multinomial-entropy},
$$
 \binom n{(N_v)_{v\in V(Q)}}
 =\binom n{nx}
 \ge \frac{2^{nh(x)}}{(n+1)^r}.
$$

We next compare $h(x)$ with $h(\delta_Q)=\lambda$.
The coordinates outside $V(Q)$ remain zero.
On $V(Q)$, the coordinates of $\delta_Q$ are positive and
$|x_v-\delta_Q(v)|\le1/n$.
Since the derivative of $-t\log_2t$ is bounded in a
neighbourhood of each of these positive coordinates,
the mean value theorem gives
$h(x)\in\lambda+O(1/n)$.
Thus, for some constant $K>0$ and all sufficiently large $n$,
we have $nh(x)\ge n\lambda-K$, and hence
$$
 \binom n{(N_v)_{v\in V(Q)}}
 \ge 2^{n\lambda-K-r\log_2(n+1)}.
$$
Choose a constant $C_1>0$ such that
$K+r\log_2(n+1)\le C_1\log_2 n$ for all sufficiently
large $n$. We obtain
\begin{equation}\label{eq:lower-colouring-count}
 \binom n{(N_v)_{v\in V(Q)}}
 \ge 2^{n\lambda-C_1\log_2 n}.
\end{equation}
For each such colouring, every pair of vertices whose
colours lie in different parts of $Q$ permits exactly one
arc, which may independently be present or absent.
No arc is permitted between vertices whose colours lie
in the same part. Thus, the number of compatible digraphs is
$2^{\sum_{1\le i<j\le\ell}N_iN_j}$.
Writing $e_i:=N_i-n/\ell$, we have $e_i=O(1)$ and
$\sum_{i=1}^{\ell}e_i=0$, so
$$
 \sum_{1\le i<j\le\ell}N_iN_j
 =dn^2-\frac12\sum_{i=1}^{\ell}e_i^2.
$$
Consequently, for some constant $C_2>0$, the number of
compatible digraphs for each colouring is at least
\begin{equation}\label{eq:lower-compatible-count}
 2^{\sum_{1\le i<j\le\ell}N_iN_j}
 \ge 2^{dn^2-C_2}.
\end{equation}

Multiplying the number of colourings in
\eqref{eq:lower-colouring-count} by the number of compatible
digraphs per colouring in \eqref{eq:lower-compatible-count}
gives
$$
 |\Col(D)_n|
 \ge
 \binom n{(N_v)_{v\in V(Q)}}
 2^{\sum_{1\le i<j\le\ell}N_iN_j}
 \ge 2^{dn^2+n\lambda-C_1\log_2 n-C_2}.
$$
For a suitable constant $C$ and all sufficiently large $n$,
this is at least $2^{dn^2+n\lambda-C\log_2 n}$, as required.
\end{proof}

Let $\mu_n$ be the distribution of the color profile
for a uniformly random element of $\Col(D)_n$.

\begin{lemma}\label{fact:one}
For every $\varepsilon>0$ we have $\mu_n(\mathcal O_\varepsilon)\to1$.
\end{lemma}
\begin{proof}
The set $\mathcal K=\mathcal Z\setminus\mathcal O_\varepsilon$ is compact,
since it is a closed subset of the compact probability simplex
$\mathcal Z$. It contains no element of $\mathcal O$. The density
function $x\mapsto d(x)$ and the entropy function $x\mapsto h(x)$
are continuous on $\mathcal Z$. Together with the compactness of
$\mathcal K$, this gives constants $\eta,c>0$ such that
\[
 x\in\mathcal K\quad\Longrightarrow\quad
 d(x)\le d-\eta\ \text{ or }\ h(x)\le\lambda-c.
\]

Fix a colour profile $x$. By \eqref{eq:col-count}, the number of $D$-coloured digraphs
with colour profile $x$ is $\binom n{nx}2^{n^2d(x)}$. 
The upper bound in \eqref{eq:multinomial-entropy-bound} gives
 $\binom n{nx}\le2^{nh(x)}$.
There are at most $(n+1)^r$ possible
colour profiles. 
Thus, the number of $D$-coloured digraphs with colour profile in $\mathcal K$ is at most
\[
 (n+1)^r\left(
  2^{(d-\eta)n^2+n\log_2 r}
  +2^{dn^2+n(\lambda-c)}\right).
\]
Divide by the lower bound from Lemma~\ref{lem:lowebound}.
Both resulting terms tend to zero.
\end{proof}

We next explain how a colour profile close to $\mathcal O$
determines an optimal support. Put \[\gamma:=1/(\ell r),  \text{ where }
r=|V(D)|.\] 
By Lemma~\ref{lem:entropy}, every positive coordinate
of an element of $\mathcal O$ has the form $1/(\ell|A_i|)$.
Since $|A_i|\le r$, all these coordinates are at least $\gamma$.

Now fix $0<\varepsilon<\gamma/4$ and $g\in\mathcal O_\varepsilon$.
By definition, there is some $\delta\in\mathcal O$ with
$\|g-\delta\|_1<\varepsilon$. In particular,
$|g_v-\delta_v|<\varepsilon$ for every $v\in V(D)$.
If $v\in\supp(\delta)$, then $g_v>\gamma-\varepsilon>\gamma/2$.
If $v\notin\supp(\delta)$, then $g_v<\varepsilon<\gamma/2$.
Thus, we can read off the support of $\delta$ directly from $g$:
\[
 \supp(\delta)=\{v\in V(D) \mid g_v>\gamma/2\}.
\]
Lemma~\ref{lem:entropy} says that an element of $\mathcal O$ is
uniquely determined by its support. Hence, the element $\delta$
within distance $\varepsilon$ of $g$ is unique. 
We write $H_g$ for the member of $S(D)$ induced by
$\{v\in V(D)\mid g_v>\gamma/2\}$.
In particular, $H_g$ depends only on $g$, not on the
choice of $\varepsilon$.

\begin{lemma}\label{lem.smalldeviations}
Let $0<\varepsilon<\gamma/4$. If a colour profile
$g\in\mathcal O_\varepsilon$ has a positive coordinate outside
$V(H_g)$, then
\[
 d(g)\le d-\frac{\gamma-3\varepsilon}{n}.
\]
\end{lemma}
\begin{proof}
If $\ell=1$, then $D$ has no arcs and $S(D)=\{D\}$, so
$V(H_g)=V(D)$ and the hypothesis is impossible. Assume $\ell\ge2$.
Choose $u\notin V(H_g)$ with $g_u\ge1/n$, and let
$A_1,\ldots,A_\ell$ be the parts of $H_g$.
There is a part, say $A_1$,  such that $u$ is non-adjacent to every vertex of $A_1$: 
otherwise, choosing
one neighbour in each part gives a tournament of order $\ell+1$.
By the maximality established in Lemma~\ref{lem:entropy}, $u$ is also non-adjacent to some vertex $w$ in a part other than $A_1$; otherwise, adjoining $u$ to $A_1$ would yield a larger induced orientation of a complete $\ell$-partite graph.
Choose $v\in A_1$, and move one vertex from colour $u$ to colour $v$:
$g'=g+(e_v-e_u)/n$, where $e_v$ is the vector with coordinate
$1$ at $v$ and $0$ at every other vertex.
Let $N(z)$ be the neighbourhood in the underlying undirected graph.
Since $u$ and $v$ are non-adjacent,
\[
 d(g')-d(g)=\frac1n\bigl(g(N(v))-g(N(u))\bigr).
\]
For the associated $\delta\in\mathcal O$, the neighbours of $v$ in
$H_g$ have total mass $(\ell-1)/\ell$, whereas those of $u$ have mass
at most $(\ell-1)/\ell-\gamma$. The $\ell_1$ error is below
$\varepsilon$, and the total $g$-mass outside $H_g$ is below
$\varepsilon$. Hence,
\[
 g(N(v))\ge(\ell-1)/\ell-\varepsilon,\qquad
 g(N(u))\le(\ell-1)/\ell-\gamma+2\varepsilon.
\]
The second bound allows an error $\varepsilon$ on the mass inside
$H_g$ and another $\varepsilon$ for the mass outside $H_g$;
this convenient bound need not be sharp.
It follows that $d(g')-d(g)\ge(\gamma-3\varepsilon)/n$.
Since $g'\in\mathcal Z$, $d(g')\le d$, proving the result.
\end{proof}

\begin{proof}[Proof of Theorem~\ref{thm:main}]
Put $\lambda_\varepsilon^+=\sup_{x\in\mathcal O_\varepsilon}h(x)$.
The continuity of $h$ gives $\lambda_\varepsilon^+ {\searrow} \lambda$ as
$\varepsilon {\searrow} 0$. It suffices to prove the theorem for
arbitrarily small positive $\varepsilon$. Choose one with
\[
 \varepsilon<\gamma/4,\qquad
 \lambda_\varepsilon^+-\lambda<\gamma-3\varepsilon.
\]
Let $B_\varepsilon$ consist of the colour profiles in
$\mathcal O_\varepsilon$ with a positive coordinate outside $H_g$.
We estimate the probability of $B_\varepsilon$ as follows:
\begin{align*}
 \mu_n(B_\varepsilon)
 &=
\sum_{g\in B_\varepsilon}
       \binom n{ng}2^{n^2d(g)}
      /|\Col(D)_n|
 &&\text{by \eqref{eq:col-count}}\\
 &\le
\sum_{g\in B_\varepsilon}
       2^{nh(g)+n^2d(g)}
      /|\Col(D)_n|
 &&\text{by \eqref{eq:multinomial-entropy-bound}}\\
 &\le
 \sum_{g\in B_\varepsilon}
       2^{n\lambda_\varepsilon^++n^2d(g)}
      /|\Col(D)_n|
 &&\text{by the definition of }\lambda_\varepsilon^+\\
 &\le
 \sum_{g\in B_\varepsilon}
       2^{n\lambda_\varepsilon^++dn^2-n(\gamma-3\varepsilon)}
      /|\Col(D)_n|
 &&\text{by Lemma~\ref{lem.smalldeviations}}\\
 &\le
 \frac{(n+1)^r
       2^{n\lambda_\varepsilon^++dn^2-n(\gamma-3\varepsilon)}}
      {|\Col(D)_n|}
 &&\text{since there are at most }(n+1)^r
   \text{ profiles}\\
 &\le
 \frac{(n+1)^r
       2^{n\lambda_\varepsilon^++dn^2-n(\gamma-3\varepsilon)}}
      {2^{dn^2+n\lambda-C\log_2 n}}
 &&\text{by Lemma~\ref{lem:lowebound}}\\
 &=
 \frac{(n+1)^r n^C}
      {2^{n(\gamma-3\varepsilon
                    -(\lambda_\varepsilon^+-\lambda))}}
                    &&\text{cancelling }2^{dn^2} \\
 &\longrightarrow 0
 &&\text{since }\gamma-3\varepsilon
                   -(\lambda_\varepsilon^+-\lambda)>0.
\end{align*}
Lemma~\ref{fact:one} excludes colour profiles outside
$\mathcal O_\varepsilon$. Every remaining colour profile is supported
on $H_g$, and each colour of $H_g$ has proportion at least
$\gamma-\varepsilon>0$. Thus, its image is exactly $V(H_g)$.
\end{proof}

\begin{corollary}\label{cor:main}
For every $\varepsilon>0$, asymptotically almost surely the image of
a random colouring induces a member $Q\in S(D)$ and every part $A_i$
of $Q$ satisfies
\[
 \sum_{v\in A_i}|f^{-1}(v)|\ge(1/\ell-\varepsilon)n.
\]
\end{corollary}
\begin{proof}
By Theorem~\ref{thm:main}, a.a.s.\ the image of $f$ induces
some $Q\in S(D)$ such that
$$
 \sum_{v\in V(D)}
 \left|\frac{|f^{-1}(v)|}{n}-\delta_Q(v)\right|
 <\varepsilon.
$$
For each part $A_i$ of $Q$, we have
$\sum_{v\in A_i}\delta_Q(v)=1/\ell$. Hence,
\begin{align*}
 \left|
 \frac1n\sum_{v\in A_i}|f^{-1}(v)|-\frac1\ell
 \right|
 &=
 \left|
 \sum_{v\in A_i}
 \left(\frac{|f^{-1}(v)|}{n}-\delta_Q(v)\right)
 \right|\\
& \le
 \sum_{v\in A_i}
 \left|\frac{|f^{-1}(v)|}{n}-\delta_Q(v)\right| < \varepsilon.
\end{align*}
Consequently,
$\sum_{v\in A_i}|f^{-1}(v)|
\ge(1/\ell-\varepsilon)n$, as required.
\end{proof}

\subsection{Comparing the optimal supports}\label{sect:weights}
Theorem~\ref{thm:main} shows that almost every colouring has an
optimal support, but does not show that each optimal support occurs
with probability bounded away from zero. The following uniform order estimate supplies this
additional information.

\begin{lemma}\label{lem:comparable}
Let $Q$ be an orientation of $K_{k_1,\ldots,k_\ell}$, with all
$k_i\ge1$ for $i\in\{1,\ldots,\ell\}$, and put $P=\prod_{i=1}^\ell k_i$, $d=(\ell-1)/(2\ell)$, and
$\lambda=\log_2\ell+\ell^{-1}\log_2P$. Then
\[
 |{\Col}(Q)_n| \in 
 \Theta\left(n^{-(\ell-1)/2}2^{dn^2+n\lambda}\right).
\]
\end{lemma}
\begin{proof}

Assigning vertices to the parts, choosing their individual
colours, and then choosing the permitted arcs gives
\begin{equation}\label{eq:multipartite-count}
 |\Col(Q)_n|=
 \sum_{\substack{n_1,\ldots,n_\ell\in {\mathbb N}\\
                  n_1+\cdots+n_\ell=n}}
 \binom n{n_1,\ldots,n_\ell}
 \left(\prod_{i=1}^{\ell}k_i^{n_i}\right)
 2^{\sum_{i<j}n_in_j}.
\end{equation}

We first estimate the largest multinomial coefficient
$B_n$. The maximum is attained when the entries differ by at most one:
if \(n_i\ge n_j+2\), moving one unit from \(n_i\) to \(n_j\) multiplies the multinomial coefficient by
\[
\frac{n_i}{n_j+1}>1.
\]
Repeatedly doing this reaches a balanced tuple. All balanced tuples have the same multinomial coefficient.

Write $n=\ell q+s$, where $0\le s<\ell$. Then
$$
 B_n=\frac{n!}{(q!)^\ell(q+1)^s}.
$$
Stirling's formula
$k!\approx\sqrt{2\pi k}\,(k/e)^k$
as $k\to\infty$,
applied to $(\ell q)!$ and $q!$, gives
$$
 B_{\ell q}
 \approx
 \frac{\sqrt{2\pi\ell q}\,(\ell q/e)^{\ell q}}
      {(2\pi q)^{\ell/2}(q/e)^{\ell q}}
 =
 \sqrt{\ell}\,(2\pi q)^{-(\ell-1)/2}\ell^{\ell q}.
$$
For each fixed $s$, we also have
$$
 \frac{B_n}{B_{\ell q}}
 =
 \frac{(\ell q+1)\cdots(\ell q+s)}{(q+1)^s}
 \longrightarrow\ell^s.
$$
Combining the preceding two estimates gives
$$
 B_n\approx
 \ell^s\sqrt{\ell}\,(2\pi q)^{-(\ell-1)/2}\ell^{\ell q}
 =
 \sqrt{\ell}\,(2\pi q)^{-(\ell-1)/2}\ell^n.
$$
Since $q\approx n/\ell$, it follows that
$$
 B_n\approx
 \frac{\ell^{\ell/2}}{(2\pi)^{(\ell-1)/2}}
 \ell^n n^{-(\ell-1)/2}.
$$
This holds along each of the finitely many possibilities
for $n\bmod\ell$, with the same constant, and therefore
holds as $n\to\infty$ without restriction. In particular,
\begin{equation}\label{eq:balanced-multinomial}
 B_n=\Theta\left(\ell^n n^{-(\ell-1)/2}\right).
\end{equation}


For each choice of non-negative integers $n_1,\ldots,n_\ell$
with $n_1+\cdots+n_\ell=n$, put $z_i:=n_i-n/\ell$ for
$i\in\{1,\ldots,\ell\}$. 
We then have 
\begin{equation}\label{eq:part-colour-factor}
 \prod_{i=1}^\ell k_i^{n_i}
 =P^{n/\ell}\prod_{i=1}^\ell k_i^{z_i}.
\end{equation}
Moreover, 
\begin{align*}
 \sum_{i=1}^\ell n_i^2
=\sum_{i=1}^\ell(n/\ell+z_i)^2
 & =\frac{n^2}{\ell}
   +\frac{2n}{\ell}\sum_{i=1}^\ell z_i
   +\sum_{i=1}^\ell z_i^2
 \\
 & =\frac{n^2}{\ell}+\sum_{i=1}^\ell z_i^2
&& \text{(since $\sum_i z_i=0$).} 
\end{align*}
Consequently, since $d=(\ell-1)/(2\ell)$,
\begin{equation}\label{eq:part-arc-exponent}
 \sum_{1\le i<j\le\ell}n_i n_j
 =\frac12\left(n^2-\sum_{i=1}^\ell n_i^2\right)
 =\frac12\left(n^2-\frac{n^2}{\ell}
                    -\sum_{i=1}^\ell z_i^2\right)
 =dn^2-\frac12\sum_{i=1}^\ell z_i^2.
\end{equation}

\emph{Lower bound.}
Choose non-negative integers $n_1,\ldots,n_\ell$ whose sum
is $n$ and which differ by at most one. For these part sizes,
$$
 \binom n{n_1,\ldots,n_\ell}=B_n
 \qquad\text{and}\qquad
 |z_i|\le1 \quad\text{for every }i\in\{1,\ldots,\ell\}.
$$
Keeping only the summand corresponding to these part sizes
in~\eqref{eq:multipartite-count}, and applying
\eqref{eq:part-colour-factor} and~\eqref{eq:part-arc-exponent},
gives
\begin{align*}
 |{\Col}(Q)_n|
 & \ge 
\binom n{n_1,\ldots,n_\ell}
 \left(\prod_{i=1}^{\ell}k_i^{n_i}\right)
 2^{\sum_{i<j}n_in_j} \\
 & = B_n P^{n/\ell}2^{dn^2}
       \left(\prod_{i=1}^\ell k_i^{z_i}\right)
       2^{-\frac12\sum_{i=1}^\ell z_i^2}.
\end{align*}
Since $k_i\ge1$ and $z_i\ge-1$, we have
$\prod_{i=1}^\ell k_i^{z_i}\ge P^{-1}$.
Also, $\sum_{i=1}^\ell z_i^2\le\ell$.
Therefore,
\begin{equation}\label{eq:multipartite-lower}
 |{\Col}(Q)_n|
 \ge \frac{B_n P^{n/\ell}2^{dn^2}}{P\,2^{\ell/2}}.
\end{equation}

\emph{Upper bound.}
We start from \eqref{eq:multipartite-count}. By the definition
of $B_n$, every multinomial coefficient in that sum is at most
$B_n$. Replacing just these coefficients gives
\[
 |\Col(Q)_n|
 \le B_n
 \sum_{\substack{n_1,\ldots,n_\ell\in\mathbb N\\
                 n_1+\cdots+n_\ell=n}}
 \left(\prod_{i=1}^{\ell}k_i^{n_i}\right)
 2^{\sum_{i<j}n_i n_j}.
\]

Next, recall that $z_i=n_i-n/\ell$. 
Substituting \eqref{eq:part-colour-factor} and \eqref{eq:part-arc-exponent} and taking the common factors 
$P^{n/\ell}$ and $2^{dn^2}$ outside the sum gives
\[
 |\Col(Q)_n|
 \le B_nP^{n/\ell}2^{dn^2}
 \sum_{\substack{n_1,\ldots,n_\ell\in\mathbb N\\
                 n_1+\cdots+n_\ell=n}}
 \left(\prod_{i=1}^{\ell}k_i^{z_i}\right)
 2^{-\frac12\sum_{i=1}^{\ell}z_i^2}.
\]
It remains to bound this last sum by a constant independent
of $n$. Denote the sum in the preceding display by $S_n$, and put
$c_i:=\log_2 k_i$ for $i\in\{1,\ldots,\ell\}$. For each $i$,
completing the square gives
\[
 k_i^{z_i}2^{-z_i^2/2}
 =2^{c_i z_i-z_i^2/2}
 =2^{c_i^2/2}\,2^{-(z_i-c_i)^2/2}.
\]
Multiplying these identities and taking the factors independent
of the summation variables outside the sum, we obtain
\[
 S_n
 =
 2^{\frac12\sum_{i=1}^{\ell}c_i^2}
 \sum_{\substack{n_1,\ldots,n_\ell\in\mathbb N\\
                 n_1+\cdots+n_\ell=n}}
 \prod_{i=1}^{\ell}2^{-(z_i-c_i)^2/2}.
\]

We next rewrite the summation conditions in terms of the $z_i$.
Since $z_i=n_i-n/\ell$, the condition $n_i\in\mathbb N$ is
equivalent to
\[
 z_i\in\mathbb Z-n/\ell
 \qquad\text{and}\qquad z_i\ge-n/\ell.
\]
Moreover, $n_1+\cdots+n_\ell=n$ is equivalent to
$z_1+\cdots+z_\ell=0$. Thus,
\[
 S_n
 =
 2^{\frac12\sum_{i=1}^{\ell}c_i^2}
 \sum_{\substack{z_1,\ldots,z_\ell\in\mathbb Z-n/\ell\\
                 z_i\ge-n/\ell\ (1\le i\le\ell)\\
                 z_1+\cdots+z_\ell=0}}
 \prod_{i=1}^{\ell}2^{-(z_i-c_i)^2/2}.
\]
All summands are non-negative. We may therefore obtain an
upper bound by dropping the restrictions $z_i\ge-n/\ell$
and $z_1+\cdots+z_\ell=0$:
\[
 S_n
 \le
 2^{\frac12\sum_{i=1}^{\ell}c_i^2}
 \sum_{z_1,\ldots,z_\ell\in\mathbb Z-n/\ell}
 \prod_{i=1}^{\ell}2^{-(z_i-c_i)^2/2}.
\]
The variables $z_1,\ldots,z_\ell$ now range independently,
and the $i$th factor depends only on $z_i$. Hence, the last
sum factors as
\[
 \sum_{z_1,\ldots,z_\ell\in\mathbb Z-n/\ell}
 \prod_{i=1}^{\ell}2^{-(z_i-c_i)^2/2}
 =
 \prod_{i=1}^{\ell}
 \left(
   \sum_{t\in\mathbb Z-n/\ell}2^{-(t-c_i)^2/2}
 \right).
\]
This identity holds first for finite truncations of the sums,
and then for the infinite sums by taking increasing limits,
since all terms are non-negative. Consequently,
\[
 S_n
 \le
 2^{\frac12\sum_{i=1}^{\ell}c_i^2}
 \prod_{i=1}^{\ell}
 \left(
   \sum_{t\in\mathbb Z-n/\ell}2^{-(t-c_i)^2/2}
 \right).
\]

We check that each series on the right is finite. Fix $i$.
For $|t|\ge\max\{2|c_i|,1\}$, we have
$|t-c_i|\ge|t|/2$, and therefore
\[
 2^{-(t-c_i)^2/2}
 \le 2^{-t^2/8}
 \le 2^{-|t|/8}.
\]
On a translate of $\mathbb Z$, the positive and negative tails
of the last expression are geometric series. The remaining
terms are finite in number, so the series converges.

Finally, these finite bounds can be chosen uniformly in $n$.
Indeed, writing $n=\ell q+s$ with $0\le s<\ell$, we have
\[
 \mathbb Z-n/\ell=\mathbb Z-s/\ell.
\]
Thus, the right-hand side of the bound for $S_n$ takes at most
$\ell$ different values, one for each residue of $n$ modulo
$\ell$. Their maximum is a finite constant $C$ independent
of $n$. We have proved $S_n\le C$, and substituting this into
our earlier bound gives
\begin{equation}\label{eq:multipartite-upper}
 |\Col(Q)_n|\le C B_nP^{n/\ell}2^{dn^2}.
\end{equation}

\ignore{
\rev{OLD: 
Every multinomial coefficient in
\eqref{eq:multipartite-count} is at most $B_n$.
Put $c_i:=\log_2 k_i$. Completing the squares gives
$$
 -\frac12\sum_{i=1}^{\ell}z_i^2
 +\sum_{i=1}^{\ell}c_i z_i
 =
 \frac12\sum_{i=1}^{\ell}c_i^2
 -\frac12\sum_{i=1}^{\ell}(z_i-c_i)^2.
$$
Each $z_i$ belongs to $\mathbb Z-n/\ell$.
Dropping the restrictions $n_i\ge0$ and $\sum_i z_i=0$
therefore yields\todo{revisit!!}
$$
 \frac{|\Col(Q)_n|}{B_nP^{n/\ell}2^{dn^2}}
 \le
 2^{\frac12\sum_{i=1}^{\ell}c_i^2}
 \prod_{i=1}^{\ell}
 \left(
 2^{\red{\sum_{t\in\mathbb Z-n/\ell}}-(t-c_i)^2/2}
 \right).
$$
Each series converges: for sufficiently large $|t|$,
its summand is at most $2^{-t^2/8}$, and hence at most
$2^{-|t|/8}$.
Moreover, the set $\mathbb Z-n/\ell$ depends only on
$n\bmod\ell$. Thus, the right-hand side is bounded by
a constant $C$ independent of $n$, giving
\begin{equation}\label{eq:multipartite-upper}
 |\Col(Q)_n|\le C B_nP^{n/\ell}2^{dn^2}.
\end{equation}}
}
Combining \eqref{eq:multipartite-lower},
\eqref{eq:multipartite-upper}, and
\eqref{eq:balanced-multinomial}, 
and since $1/(P\,2^{\ell/2})$
is a constant independent of $n$, 
we obtain
$$
 |\Col(Q)_n|
\in
 \Theta\left(
 n^{-(\ell-1)/2}\ell^nP^{n/\ell}2^{dn^2}
 \right)
 =
 \Theta\left(
 n^{-(\ell-1)/2}2^{dn^2+n\lambda}
 \right),
$$
where the last equality uses
$\lambda=\log_2\ell+\ell^{-1}\log_2P$.

\end{proof}

\begin{corollary}\label{cor:optimal-weights}
For every $Q\in S(D)$, put
$\mathcal E_Q^n=\{(G,f)\in\Col(D)_n \mid \operatorname{im}f=V(Q)\}$.
Then
\begin{align}
 |{\Col}(D)_n| & \in (1+o(1))\sum_{Q\in S(D)}|{\Col}(Q)_n|, \label{eq:ColD} \\
 |\mathcal E_Q^n| & =(1+o(1))|{\Col}(Q)_n|, \label{eq:EQn} \\ 
\text{ and } \quad 
0 & < \liminf_n |\mathcal E_Q^n|/|{\Col}(D)_n|. 
\end{align} 
\end{corollary}
\begin{proof}
Applied to $Q$, Theorem~\ref{thm:main} says that almost every
$Q$-colouring is surjective, giving the second equality. The sets
$\mathcal E_Q^n$ are disjoint, and applied to $D$ the same theorem
says their union has probability tending to one. 
It remains to show that each optimal support has probability
bounded away from zero.
All members of $S(D)$ have the same $\ell$ and $P$, and hence
the same $\lambda$. Put
$$
 F_n:=n^{-(\ell-1)/2}2^{dn^2+n\lambda}.
$$
By Lemma~\ref{lem:comparable}, for each $Q\in S(D)$ there
are constants $c_Q,C_Q>0$ such that
$$
 c_QF_n\le|\Col(Q)_n|\le C_QF_n
$$
for all sufficiently large $n$.
Then 
$$
 \frac{|\Col(Q)_n|}
      {\sum_{Q'\in S(D)}|\Col(Q')_n|}
 \ge
 \frac{c_QF_n}{\sum_{Q'\in S(D)}C_{Q'}F_n}
 =
 \frac{c_Q}{\sum_{Q'\in S(D)}C_{Q'}}>0.
$$
The two asymptotic equalities \eqref{eq:EQn} and \eqref{eq:ColD} proved above therefore imply
\begin{align*}
 \liminf_{n\to\infty}
 \frac{|\mathcal E_Q^n|}{|\Col(D)_n|}
 & \ge
 \frac{c_Q}{\sum_{Q'\in S(D)}C_{Q'}}>0.
 && \qedhere
\end{align*}
\end{proof}

\subsection{Model theory of \texorpdfstring{$D$}{D}-coloured digraphs}\label{sect:D-col}
We now pass from counting coloured digraphs to describing their
almost-sure theories. Recall that $M^{\mathrm{col}}(D)$ denotes the
generic $D$-coloured digraph from Lemma~\ref{lem:AP}, while $M(D)$
denotes its uncoloured reduct. For complete multipartite orientations,
concentration gives colour classes large enough to realize each
extension axiom with probability tending to one. We then combine
this fact with concentration on optimal supports to treat general
finite oriented templates.

\begin{proposition}\label{prop:Kl}
If $D$ is an orientation of a complete $\ell$-partite graph, then
$T_{\Col(D)}=\Th(M^{\mathrm{col}}(D))$.
\end{proposition}

\begin{proof}
Here $S(D)=\{D\}$, so Theorem~\ref{thm:main} implies that for some
$\beta>0$, asymptotically almost surely every colour class has size
at least $\beta n$. Fix a colouring with this property.
Conditioned on it, each permitted arc is present independently with
probability $1/2$.

Consider a $D$-colored one-point extension axiom, given by a $D$-coloured digraph $B$ of size $m+1$ and $b \in B$. Let $A$ be the substructure of $B$ with domain $B \setminus \{b\}$. 
Let $u$ be the colour of $b$. There are at
least $\beta n-m$ candidates of colour $u$ outside $A$.
Each candidate has probability at least $2^{-m}$ of realizing the
required diagram: each of at most $m$ permitted adjacencies
prescribes one independent binary decision, and forbidden
adjacencies are absent automatically. These events are independent
for distinct candidates, even after conditioning on the tuple's
diagram. Thus, the probability that this axiom fails is at most
\[
 n^m(1-2^{-m})^{\beta n-m}\longrightarrow0.
\]
The universal axioms for $D$-coloured structures always hold.
The preceding bound is uniform over colourings whose classes
all have size at least $\beta n$, and such colourings have probability
tending to one. Hence, each fixed extension axiom holds
asymptotically almost surely, without conditioning on the colouring.

Let $\Sigma_D$ consist of the universal axioms and all admissible
one-point extension axioms. If a sentence $\varphi$ is true in
$M^{\mathrm{col}}(D)$, Lemma~\ref{lem:ext-ax} gives
$\Sigma_D\models\varphi$. Thus, $\Sigma_D\cup\{\neg\varphi\}$ has
no model. By the compactness theorem, there is a finite subset
$\Sigma_0\subseteq\Sigma_D$ such that $\Sigma_0\models\varphi$.
Each axiom in $\Sigma_0$ holds with probability tending to one;
the union bound over this finite set shows that they hold
simultaneously with probability tending to one. Consequently,
$\varphi$ holds asymptotically almost surely. If $\varphi$ is false
in $M^{\mathrm{col}}(D)$, apply the same argument to $\neg\varphi$
to see that the probability of $\varphi$ tends to zero. This proves
$T_{\Col(D)}=\Th(M^{\mathrm{col}}(D))$.
\end{proof}

\begin{theorem}\label{thm:Col}
If $S(D)=\{Q\}$, then $\Col(D)$ has the $\omega$-categorical
almost-sure theory of the generic $Q$-coloured digraph, expanded by
empty predicates for the colours outside $Q$.
If $|S(D)|>1$, its almost-sure theory is incomplete.
\end{theorem}
\begin{proof}
In the first case, Theorem~\ref{thm:main} and
Proposition~\ref{prop:Kl} apply; conditioning on a surjective
$Q$-colouring changes probabilities by $o(1)$.
In the second, choose distinct $Q_1,Q_2\in S(D)$ and
$u\in V(Q_1)\setminus V(Q_2)$, which exists by inclusion-wise
maximality. We reuse the terminology from Corollary~\ref{cor:optimal-weights}
where $\mathcal E_Q^n$ is the set of pairs $(G,f)\in\Col(D)_n$
whose colouring has image exactly $V(Q)$. The sentence $\exists x.\,P_u(x)$ is true on
$\mathcal E_{Q_1}^n$ and false on $\mathcal E_{Q_2}^n$.
Both events have probabilities bounded away from zero by
Corollary~\ref{cor:optimal-weights}. Neither this sentence nor its
negation is almost surely true.
\end{proof}

\subsection{Removing the colouring}\label{sect:forget}
Forgetting the colouring does not in general give a uniform
uncoloured digraph: the fibre over $G$ has size $|\Hom(G,D)|$.
Here the \emph{fibre over $G$} is the set of coloured pairs
$\{(G,f) \mid f\in\Hom(G,D)\}$ mapped to $G$ by forgetting the colouring.
Thus, uniform sampling of coloured pairs weights $G$ in proportion
to its number of homomorphisms to $D$.
In this section we discuss how to get rid of the coloring. We will use the following lemma.

\begin{lemma}\label{lem:a}
Let $D$ be a finite non-empty digraph. Suppose that, for a fixed
positive integer $a$, a uniformly random $(G,f)\in\Col(D)_n$
satisfies $|\Hom(G,D)|=a$ with probability tending to one.
Then coloured and uncoloured probabilities differ by $o(1)$ for
every property of the uncoloured digraph, and
\[
 |\Csp(D)_n|\approx |{\Col}(D)_n|/a.
\]
\end{lemma}
\begin{proof}
Put $\mathcal G_n=\{G\in\Csp(D)_n \mid |\Hom(G,D)|=a\}$ and let
$\mathcal B_n$ be its complement. Writing
$b_n=\sum_{G\in\mathcal B_n}|\Hom(G,D)|$, the hypothesis gives
\[
 |{\Col}(D)_n|=a|\mathcal G_n|+b_n,
 \qquad b_n=o(|\mathcal G_n|).
\]
Since $|\mathcal B_n|\le b_n$, the good set has asymptotic probability
one under both measures. On that set all fibres have size $a$, so
the induced uncoloured measures coincide exactly.
\end{proof}

Two distinct vertices are \emph{twins} if they are non-adjacent and
have the same in-neighbours and the same out-neighbours.
A digraph is \emph{twin-free} if it has no such pair; a
\emph{twin class} is an equivalence class for the relation
``equal or twins''.

\begin{lemma}\label{lem:twins}
For an orientation $Q$ of a complete multipartite graph the following
are equivalent:
\begin{enumerate}
\item\label{item:twins-core} {$Q$ is a core;}
\item\label{item:twins-blowup} {$Q$ is not a blow-up of a proper subgraph;}
\item\label{item:twins-free} $Q$ is twin-free.
\end{enumerate}
Moreover, every homomorphism from a
twin-free such $Q$ to an oriented graph is injective.
\end{lemma}
\begin{proof}
$(\ref{item:twins-core})\Rightarrow(\ref{item:twins-blowup})$.
If $Q$ is a blow-up of a proper subgraph $H$, mapping each vertex
of $Q$ to the vertex of $H$ that it duplicates gives a homomorphism
$Q\to H$. Together with the inclusion $H\to Q$, this shows that
$Q$ is not a core.

$(\ref{item:twins-blowup})\Rightarrow(\ref{item:twins-free})$.
If $x,y$ are twins in $Q$, then $Q$ can be obtained from the proper
induced subgraph $Q-y$ by duplicating $x$. Thus, $Q$ is a blow-up of
a proper subgraph.

$(\ref{item:twins-free})\Rightarrow(\ref{item:twins-core})$.
Suppose that $Q$ is twin-free
and $f \colon Q\to H$ is a homomorphism to an oriented graph $H$.
Adjacent vertices cannot have the same image.
If $x,y$ are distinct non-adjacent vertices of $Q$, they lie in the
same part. Since they are not twins, some vertex $z$ in another
part has opposite orientations towards them: either
$x\longrightarrow z\longrightarrow y$ or $y\longrightarrow z\longrightarrow x$.
If $f(x)=f(y)$, there is no possible image for $z$. Hence, $f$ is injective.
In particular, every endomorphism of the finite digraph $Q$ is
injective and therefore an automorphism, so $Q$ is a core.
\end{proof}

By \emph{colour recovery up to an automorphism of $D$} we mean
that any two homomorphisms from the same digraph to $D$ differ only
by postcomposition with an automorphism of $D$. The following
lemma gives a first-order condition, in the uncoloured language,
that guarantees this property.

\begin{lemma}[First-order colour recovery]\label{lem:recovery}
Let $D$ be a finite non-empty core digraph. Then there is a
first-order $\{E\}$-sentence $\sigma_D$ such that
$M(D)\models\sigma_D$ and, whenever $G$ is a digraph with a homomorphism $f \colon G\to D$ and
$G\models\sigma_D$, the map $f$ is surjective and
\[
 \Hom(G,D)=\{\alpha\circ f \mid \alpha\in\Aut(D)\}.
\]
\end{lemma}
\begin{proof}

We use the observation that an endomorphism of a finite core that fixes every vertex except possibly $d$ must also fix $d$: otherwise it identifies $d$ with another, already fixed, vertex.
For a tuple $\bar a=(a_d)_{d\in V(D)}$ let $\chi_D(\bar a)$
state that the entries of $\bar a$ 
are pairwise distinct and induce a copy of $D$. 
The following construction is illustrated in Figure~\ref{fig:colour-recovery-cycle}.
For $u\in V(D)$ put
\[
\begin{split}
 \eta_u(\bar a,x,\bar b):={}&\chi_D(\bar b)\ \wedge
 \bigwedge_{(j,k)\in E(D)}
       \big (E(a_j,b_k)\wedge E(b_j,a_k) \big)\\
 &{}\wedge\bigwedge_{(u,k)\in E(D)}E(x,b_k)
       \wedge\bigwedge_{(k,u)\in E(D)}E(b_k,x).
\end{split}
\]

Let $G$ be a graph, let $x \in V(G)$,  and let $\bar a=(a_d)_{d\in V(D)}$ and $\bar b=(b_d)_{d\in V(D)}$ 
be tuples of vertices of $G$ such that  
$G \models \eta_u(\bar a,x,\bar b)$. If $q \colon G \to D$ is a homomorphism, 
then $q$ is surjective, because $D$ is a core, 
and the map $\beta_q \colon V(D) \to V(D)$ defined by
$\beta_q(d) := q(b_d)$ is an automorphism of $D$. 
Normalize $q$ by postcomposing with $\beta_q^{-1}$, so $q(b_d)=d$.


For each $d \in V(D)$, define $s_d \colon V(D) \to V(D)$ by $$
s_d(v) := \begin{cases} q(a_d) & \text{ if }v=d \\
v & \text{otherwise.} \end{cases}$$
The required arcs between $\bar a$ and $\bar b$ imply that $s_d$ is an endomorphism of $D$. Since $D$ is a core and $s_d$ fixes every vertex other than possibly $d$, it must also fix $d$. Thus, $q(a_d) = d$. The same argument for the map sending $u$ to $q(x)$ and fixing every other vertex gives $q(x) = u$.

Undoing the normalization gives
\begin{equation}\label{eq:fixation}
 q(a_d)=\beta_q(d)\quad(d\in V(D)),\qquad q(x)=\beta_q(u).
\end{equation}
Define
\begin{align}
 \sigma_D:=\exists\bar a\left(\chi_D(\bar a)\wedge
       \forall x\bigvee_{u\in V(D)}\exists\bar b\,
              \eta_u(\bar a,x,\bar b)\right).
              \label{def:sigma}
\end{align}

\begin{figure}[htbp]
\centering
\begin{tikzpicture}[
  vertex/.style={circle,draw,fill=white,minimum size=8mm,
                 inner sep=1pt},
  arc/.style={->,>=stealth,semithick},
  crossarc/.style={arc,blue!65!black,dashed},
  xarc/.style={arc,orange!80!black,very thick}
]
  \node[vertex] (b1) at (0,2) {$b_1$};
  \node[vertex] (b2) at (-1.75,-1) {$b_2$};
  \node[vertex] (b3) at (1.75,-1) {$b_3$};
  \node[vertex] (a1) at (0,3.2) {$a_1$};
  \node[vertex] (a2) at (-2.85,-1.65) {$a_2$};
  \node[vertex] (a3) at (2.85,-1.65) {$a_3$};
  \node[vertex,fill=orange!12] (x) at (0,0) {$x$};

  \draw[arc] (b1) -- (b2);
  \draw[arc] (b2) -- (b3);
  \draw[arc] (b3) -- (b1);

  \draw[crossarc] (a1) -- (b2);
  \draw[crossarc] (b1) -- (a2);
  \draw[crossarc] (a2) -- (b3);
  \draw[crossarc] (b2) -- (a3);
  \draw[crossarc] (a3) -- (b1);
  \draw[crossarc] (b3) -- (a1);

  \draw[xarc] (x) -- (b2);
  \draw[xarc] (b3) -- (x);
\end{tikzpicture}
\caption{The formula $\eta_1(\bar a,x,\bar b)$ for
$D=\vec C_3$, with arcs $1\to2\to3\to1$.
The vertices $b_1,b_2,b_3$ induce the black directed cycle,
as required by $\chi_D(\bar b)$.
The six dashed blue arcs are the constraints between
$\bar a$ and $\bar b$; the two thick orange arcs are the
constraints involving $x$.
All eleven required arcs are shown.
Under every homomorphism to $D$, the vertices $a_1,b_1,x$
have the same image.}
\label{fig:colour-recovery-cycle}
\end{figure}

We claim that $M(D) \models \sigma_D$.  
Choose a colour-respecting copy
$\bar a$ of $D$ in $M(D)$. For any vertex $x$ of $M(D)$ of colour $u$,  by the extension property of $M(D)$ we can find vertices $\bar b$ such that 
$M(D) \models \eta_u(\bar a,x,\bar b)$. Hence, $M(D)\models\sigma_D$.

Finally, suppose that $G$ is a graph such that $G\models\sigma_D$ 
and choose a tuple $\bar a$ satisfying
$$
 \chi_D(\bar a)\wedge
 \forall x\bigvee_{u\in V(D)}\exists\bar b\,
 \eta_u(\bar a,x,\bar b).
$$
For every $q\colon G\to D$, the map $\rho_q \colon V(D) \to V(D)$ defined by $\rho_q(d):=q(a_d)$ is an automorphism, which
also proves that $q$ is surjective. 
For each $x\in V(G)$, choose $u\in V(D)$ such that
$G\models\exists\bar b\,\eta_u(\bar a,x,\bar b)$.
Equation~\eqref{eq:fixation} gives
$q(x)=\rho_q(u)$ for every $q\colon G\to D$. Applying this to $f$ and an arbitrary
$h\colon G\to D$ yields $h=(\rho_h\circ\rho_f^{-1})\circ f$. This shows that 
$\Hom(G,D) \subseteq \{\alpha\circ f \mid \alpha\in\Aut(D)\}$.  
The converse inclusion is immediate.
\end{proof}

\begin{proposition}\label{prop:no-twins}
If $Q$ is a core orientation of a complete multipartite graph, then
\[
 T_{\Csp(Q)}=\Th(M(Q)),\qquad
 |\Csp(Q)_n|\approx\frac{|{\Col}(Q)_n|}{|\Aut(Q)|}.
\]
\end{proposition}
\begin{proof}
By Proposition~\ref{prop:Kl}, almost every $Q$-coloured digraph
satisfies $\sigma_Q$. Lemma~\ref{lem:recovery} gives exactly
$|\Aut(Q)|$ colourings for such a digraph, since each colouring is
surjective. Apply Lemma~\ref{lem:a}.
\end{proof}

\begin{lemma}\label{lem:back-and-forth}
If $D'$ is a blow-up of a finite non-empty oriented graph $D$, then
$M(D')\cong M(D)$.
\end{lemma}
\begin{proof}
Collapse the twin colours of the generic $D'$-coloured structure
according to the blow-up map $D'\to D$.
Every admissible one-point $D$-coloured extension can be realized:
choose any colour of $D'$ above the new point's $D$-colour and use
the extension property for $D'$. The required adjacencies depend
only on the collapsed colours. The resulting $D$-coloured structure
therefore satisfies Lemma~\ref{lem:ext-ax}, and is the generic one.
Taking reducts proves the assertion.
\end{proof}

\begin{proposition}\label{prop:no-core}
If $Q$ is any orientation of a complete multipartite graph, then
$T_{\Csp(Q)}=\Th(M(Q))$.
\end{proposition}
\begin{proof}
Collapse each twin class to obtain a twin-free induced subgraph $Q'$.
Then $Q$ is a blow-up of $Q'$ and $\Csp(Q)=\Csp(Q')$.
Lemma~\ref{lem:twins} makes $Q'$ a core. Apply
Proposition~\ref{prop:no-twins} and Lemma~\ref{lem:back-and-forth}.
\end{proof}

\subsection{Almost-sure theories of finite-domain CSPs}\label{sect:results}
We now combine concentration on optimal supports with the
colour-recovery results to identify the almost-sure theory explicitly.
The following stronger version also gives an asymptotic enumeration;
Theorem~\ref{thm:main-csp} will follow as a consequence.

\begin{theorem}\label{thm:csp-theories}
Let $D$ be a finite non-empty oriented graph, and let $S(D)$ be
as in Theorem~\ref{thm:main-csp}.
If $S(D)$ contains a core $Q$ and every core in $S(D)$ is
isomorphic to $Q$, then
\[
 |\Csp(D)_n|\approx |\Csp(Q)_n|
 \approx \frac{|{\Col}(Q)_n|}{|\Aut(Q)|},
 \qquad
 T_{\Csp(D)}=\Th(M(Q)),
\]
where $M(Q)$ is the uncoloured reduct of the countable homogeneous
$Q$-coloured structure $M^{\mathrm{col}}(Q)$.
If $S(D)$ contains two non-isomorphic cores, then the almost-sure
theory of $\Csp(D)$ is incomplete.
\end{theorem}

\begin{proof}
Recall the conclusion of Theorem~\ref{thm:main}: for every fixed
$\varepsilon>0$, almost every $(G,f)\in\Col(D)_n$ has
$\operatorname{im}f=V(A)$ for some $A\in S(D)$ and satisfies
$\sum_{v\in V(D)}\bigl||f^{-1}(v)|-n\delta_A(v)\bigr|<\varepsilon n$.
Put $\gamma=1/(\ell r)$, where $r=|V(D)|$;
every positive coordinate of $\delta_A$ is at least $\gamma$.
Fix $0<\varepsilon<\gamma/4$ and call a colouring \emph{typical}
if it satisfies the stated image and size conditions.
There are $o(|{\Col}(D)_n|)$ atypical coloured pairs, and every
colour class in the image of a typical colouring has at least
$(\gamma-\varepsilon)n$ vertices.

We also record a counting estimate used in both cases of the theorem.
For every core $Q\in S(D)$, Proposition~\ref{prop:no-twins} and
Corollary~\ref{cor:optimal-weights} give
\begin{equation}\label{eq:core-comparison}
 |\Csp(Q)_n|\approx\frac{|{\Col}(Q)_n|}{|\Aut(Q)|}
              =\Theta(|{\Col}(D)_n|).
\end{equation}

First suppose that $S(D)$ contains a core $Q$ and that all its
cores are isomorphic to $Q$. We will show that the digraphs in
$\Csp(D)_n\setminus\Csp(Q)_n$ are negligible by separating them
according to whether they admit a typical colouring.
If such a digraph $G$ has no typical colouring, choose any
homomorphism $f\colon G\to D$, which exists because $G\in\Csp(D)_n$.
The pair $(G,f)$ is atypical. Different digraphs give different
coloured pairs, so the number of digraphs in this case is
$o(|{\Col}(D)_n|)$.

Now suppose that $G\in\Csp(D)_n\setminus\Csp(Q)_n$ admits a typical
colouring $f$ with image $V(A)$, where $A\in S(D)$.
The graph $A$ cannot be a core: otherwise the hypothesis gives
$A\cong Q$, and composing $f\colon G\to A$ with this isomorphism
would give a homomorphism $G\to Q$, contradicting
$G\notin\Csp(Q)_n$.
By Lemma~\ref{lem:twins}, $A$ therefore has twins $u,v$.
Typicality gives $|f^{-1}(u)|,|f^{-1}(v)|\ge(\gamma-\varepsilon)n$.
We may reassign every vertex of $f^{-1}(u)\cup f^{-1}(v)$
independently to either $u$ or $v$: these colours are non-adjacent
and have identical in- and out-neighbours, so every reassignment
is still a homomorphism to $A$, and hence to $D$.
Thus, $G$ has at least $2^{2(\gamma-\varepsilon)n}$ distinct
$D$-colourings. These new colourings need not be typical.
Since colourings of different digraphs give disjoint sets of
coloured pairs, the number of digraphs in this case is at most
\[
 2^{-2(\gamma-\varepsilon)n}|{\Col}(D)_n|.
\]
The factor $2^{-2(\gamma-\varepsilon)n}$ tends to zero.
Combining the two cases and using \eqref{eq:core-comparison} yields
\[
 |\Csp(D)_n\setminus\Csp(Q)_n|=o(|\Csp(Q)_n|).
\]
Since $Q$ is a subgraph of $D$, we have
$\Csp(Q)_n\subseteq\Csp(D)_n$, so this proves
$|\Csp(D)_n \approx |\Csp(Q)_n|$.
The second asymptotic equivalence is already contained in
\eqref{eq:core-comparison}. Lemma~\ref{lem:transfer} gives
$T_{\Csp(D)}=T_{\Csp(Q)}$, and Proposition~\ref{prop:no-twins}
identifies this theory with $\Th(M(Q))$.

For the second case, let $A,B\in S(D)$ be non-isomorphic cores.
We first show that neither maps homomorphically to the other.
Suppose that $h\colon A\to B$ is a homomorphism.
By the final assertion in Lemma~\ref{lem:twins}, $h$ is injective.
Fix one part of $A$ and choose one vertex in each of its other
$\ell-1$ parts. The chosen vertices are pairwise adjacent, so
their images occupy distinct parts of $B$. Every vertex of the
fixed part is adjacent to all the chosen vertices; its image
must therefore lie in the one remaining part of $B$.
Thus, $h$ maps each part of $A$ into a single part of $B$.
Different parts must map into different parts because all vertices
in different parts are adjacent.
After indexing the parts accordingly, injectivity gives
$|A_i|\le|B_i|$ for $1\le i\le\ell$.
Membership in $S(D)$ gives
$\prod_{i=1}^\ell|A_i|=\prod_{i=1}^\ell|B_i|$, so all these
inequalities are equalities. Hence, $h$ is bijective on every part.
Between distinct parts there is exactly one arc for each pair of
vertices, and $h$ preserves its direction. Within parts there are
no arcs. Consequently, $h$ is an isomorphism, a contradiction.
Interchanging $A$ and $B$ also rules out a homomorphism $B\to A$.

Applying \eqref{eq:core-comparison} to $A$ and $B$ shows that both
$|\Csp(A)_n|$ and $|\Csp(B)_n|$ are of order $|{\Col}(D)_n|$.
Every digraph in $\Csp(D)_n$ has at least one $D$-colouring, so
$|\Csp(D)_n|\le|{\Col}(D)_n|$. It follows that each of the subclasses
$\Csp(A)_n$ and $\Csp(B)_n$ has proportion bounded away from zero
in $\Csp(D)_n$.
The generic structure $M(B)$ contains a copy of $B$ and therefore
satisfies its conjunctive query $\theta_B$.
By Proposition~\ref{prop:no-twins}, almost every digraph in
$\Csp(B)_n$ consequently satisfies $\theta_B$.
In contrast, no $G\in\Csp(A)_n$ satisfies $\theta_B$, since this
would give homomorphisms $B\to G\to A$, which we have ruled out.
Thus, both $\theta_B$ and its negation have probabilities bounded
away from zero in $\Csp(D)_n$, proving that its almost-sure theory
is incomplete.
\end{proof}

\begin{proof}[Proof of Theorem~\ref{thm:main-csp}]
In the first case, Theorem~\ref{thm:csp-theories} identifies the
almost-sure theory with $\Th(M(Q))$, which is complete.
It is $\omega$-categorical because $M(Q)$ is a reduct of the
countable homogeneous structure $M^{\mathrm{col}}(Q)$ in a finite
relational signature. In the second case, homomorphically
inequivalent cores are non-isomorphic, so the second assertion of
Theorem~\ref{thm:csp-theories} applies.
\end{proof}

\begin{example}\label{expl:combine}
For $D=\vec C_3\uplus T_3$ from Example~\ref{expl:no-01}, the
limiting probability of the conjunctive query of $T_3$ is $3/4$.
Indeed, the two coloured counts are equal, since the number of
permitted arcs depends only on the colour-class sizes. However, the symmetries of $\vec C_3$ leads to over-counting and therefore there will be more non-isomorphic digraphs mapping to $T_3$. The rest of this discussion is a formalisation of this idea.

By Theorem~\ref{thm:main}, a.a.s.\ the colouring has image one of
the two triangles. For $H\in\{T_3,\vec C_3\}$, put
$\mathcal U_H^n=\{G\in\Csp(H)_n \mid G\models\sigma_H\}$.
Proposition~\ref{prop:Kl} shows that a.a.s.\ an $H$-coloured
digraph has its uncoloured reduct in $\mathcal U_H^n$.

If $G\in\mathcal U_H^n$, then for each $x\in V(G)$ the formula
$\sigma_H$ from Lemma~\ref{lem:recovery} supplies a tuple $\bar b$ and a colour $u$ satisfying
$\eta_u(\bar a,x,\bar b)$.
The tuple $\bar b$ induces a copy of $H$, and $x$ is adjacent
to at least one of its vertices because $u$ has a neighbour in
$H$. Thus, $x$ and that copy lie in a connected subgraph. Any
homomorphism from $G$ to $D$ maps this subgraph into one component
of $D$. The two triangles are homomorphically incomparable, so
that component must be $H$. As this holds for every $x$, all
homomorphisms $G\to D$ have image in $H$.

Lemma~\ref{lem:recovery} now gives exactly $|\Aut(H)|$
$D$-colourings of each digraph in $\mathcal U_H^n$. These fibre
sizes are $1$ for $T_3$ and $3$ for $\vec C_3$, and the two sets
$\mathcal U_H^n$ are disjoint. Coloured pairs whose reducts lie
outside their union number $o(|{\Col}(D)_n|)$, by concentration
and Proposition~\ref{prop:Kl}; the number of such uncoloured
digraphs is no larger. Dividing the equal leading coloured counts
by the respective fibre sizes therefore gives uncoloured
proportions in the ratio $1:1/3$. The query of $T_3$ holds
on $\mathcal U_{T_3}^n$ and fails on $\mathcal U_{\vec C_3}^n$,
so its limiting probability is $3/4$.
This is the same fibre-counting argument used in full detail in
Proposition~\ref{prop:uncolour} below.
\end{example}

\section{Non-convergence}\label{sect:non-convergence}
We now construct a finite core oriented graph $D$ for which
$\Csp(D)$ does not have first-order convergence. Our template $D$
is the disjoint union of two bipartite tournaments whose part sizes
have the same product. The coloured counts of these two components 
have the same quadratic and linear exponential
terms, but their ratio has different limits along even and odd
orders. We first compute the counts and separate the two limits;
we then use colour recovery to obtain an uncoloured example.

\subsection{Bipartite tournaments}

Call an oriented graph a \emph{bipartite tournament} if it is an
orientation of a complete bipartite graph.  If $M$ is a $0$--$1$ matrix
with rows indexed by $L$ and columns indexed by $R$, let $T(M)$ be the
bipartite tournament with parts $L,R$ in which
for all $\ell \in L$ and $r \in R$  
\[
  \ell \longrightarrow r \iff M_{\ell,r}=1,
  \qquad
  r \longrightarrow \ell \iff M_{\ell,r}=0 .
\]
Every bipartite tournament with parts $L,R$ arises this way, and $T(M)$
is twin-free in the sense of Section~\ref{sect:forget} if and only if the rows of $M$
are pairwise distinct and the columns of $M$ are pairwise distinct.  We
write $\vec K_{a,b}$ for $T(M)$ with $M$ the all-ones $a\times b$
matrix, i.e., for the complete bipartite graph $K_{a,b}$ with all arcs
oriented from the side of size $a$ to the side of size $b$.

\begin{example}\label{expl:S(D)}
For $i=1,2$ let $C_i$ be a bipartite tournament whose parts have sizes
$a_i,b_i\ge1$, and suppose that $a_1b_1=a_2b_2=:p$.  Put
$D:=C_1\uplus C_2$. Then $D$ is an oriented graph and the largest
order $\ell$ of a tournament contained in $D$ is $2$.
Moreover, $S(D)$ (defined in Theorem~\ref{thm:main-csp}) 
equals $\{C_1,C_2\}$.
\end{example}


\subsection{Exact counting and lattice asymptotics}
We now count the $D$-coloured digraphs with vertex set $[n]$ if $D$ is is a bipartite tournament. 
First we obtain an exact formula, which depends
only on the sizes of the two parts, not on the orientation of the
arcs. We then determine its asymptotic behaviour separately for even
and odd $n$. When two supports have the same product of part sizes,
the leading factors in their counts cancel in the ratio. What remains
is a ratio of constants that may depend on 
the
parity of $n$. The next subsection
provides a criterion for the even and odd ratios to be different. The following is a consequence of~\eqref{eq:multipartite-count} for $\ell = 2$.

\begin{lemma}\label{lem:Z-exact}
Let $C$ be a bipartite tournament with parts of sizes $a$ and $b$.  Then
$|{\Col}(C)_n|=Z_{a,b}(n)$, where
\begin{equation}\label{eq:Zsum}
  Z_{a,b}(n) \;:=\; \sum_{k=0}^{n}\binom{n}{k}a^{k}b^{\,n-k}2^{\,k(n-k)} .
\end{equation}
In particular, $|{\Col}(C)_n|$ depends only on $a$ and $b$, not on the orientation.
\end{lemma}


\begin{remark}
For $a=b=1$, this is Read's count $M_n(2)$ of labelled
graphs that are properly two-coloured 
(allowing unused colours); see~\cite[Section~3, equation~(7),
p.~413]{Read1960}.
Wright~\cite{Wright1961} derived asymptotic expressions for Read's formula; our treatment below is inspired by this (also see Remark~\ref{rem:wright1}).
\end{remark}

In the sum in~\eqref{eq:Zsum}, substitute $k=n/2+j$.
Then $j$ ranges over $\mathbb Z-n/2$ with $-n/2\le j\le n/2$.
This is where the parity of $n$ enters, since the set of indices 
is $\mathbb Z$ for even $n$ and $\mathbb Z+\frac12$ for odd $n$.  For
$\nu\in\{0,\frac12\}$ and $r>0$ put
\begin{equation}\label{eq:theta}
  \Psi_\nu(r) \;:=\; \sum_{j\in\mathbb Z+\nu} r^{\,j}\,2^{-j^{2}} ,
\end{equation}
an absolutely convergent series, and let $\nu_n:=\frac
n2-\lfloor\frac n2\rfloor\in\{0,\frac12\}$.

\begin{remark}\label{rem:wright1}
    The next lemma is a weighted analogue of the two-colour case
of Wright's asymptotic expansion~\cite[Theorem~1, Equation~(2.1),
p.~685]{Wright1961}. For $a=b=1$, the centred rewriting
used in its proof is the specialization to two colours of
\cite[equation~(4.1), p.~688]{Wright1961}; the weights $a,b$
introduce the additional factor $p^{n/2}r^j$.
We give a self-contained proof.
\end{remark}

\begin{lemma}\label{lem:Z-asym}
Let $a,b\ge1$, $p:=ab$ and $r:=a/b$.  Then
\[
  Z_{a,b}(n)
  \;=\;
  \binom{n}{\lfloor n/2\rfloor}\,p^{\,n/2}\,2^{\,n^{2}/4}
  \bigl(\Psi_{\nu_n}(r)+o(1)\bigr)
  \qquad(n\to\infty).
\]
\end{lemma}

\begin{proof}
Put $\mathcal J_n:=\{k-\frac n2 \mid 0\le k\le n\}$, so that
$\mathcal J_n=(\mathbb Z+\nu_n)\cap[-\frac n2,\frac n2]$.  For
$k=\frac n2+j$ we have $k(n-k)= (\frac{n}{2} + j)(\frac{n}{2} - j) = \frac{n^2}{4}-j^{2}$ and
$a^{k}b^{\,n-k}=a^{n/2+j} b^{n/2 - j} =(ab)^{n/2}\left(\frac ab\right)^j =  p^{\,n/2}r^{\,j}$, 
so \eqref{eq:Zsum} becomes
\begin{align*}
    Z_{a,b}(n)= \sum_{k=0}^{n}\binom{n}{k}a^{k}b^{\,n-k}2^{\,k(n-k)} = \binom{n}{\lfloor n/2\rfloor}p^{\,n/2}2^{\,n^{2}/4} \left (\sum_{j\in\mathcal J_n}\beta_n(j)\,r^{\,j}2^{-j^{2}} \right ),
\end{align*}
where 
$\beta_n(j):=\binom{n}{\frac n2+j}\big/\binom{n}{\lfloor
n/2\rfloor}$. Note that $\beta_n(j) \in(0,1]$, because $\binom{n}{\lfloor n/2\rfloor}$ is the
largest binomial coefficient in row $n$.  Fix $\nu\in\{0,\frac12\}$ and
$\eta>0$, and choose a finite $F\subseteq\mathbb Z+\nu$ with
$\sum_{j\in(\mathbb Z+\nu)\setminus F}r^{\,j}2^{-j^{2}}<\eta$.  For all
$n$ with $\nu_n=\nu$ and $F\subseteq\mathcal J_n$,
\[
  \bigl|\sum_{j\in\mathcal J_n}\beta_n(j)\,r^{\,j}2^{-j^{2}}-\Psi_\nu(r)\bigr|
  \;\le\;
  \sum_{j\in F}\bigl(1-\beta_n(j)\bigr)r^{\,j}2^{-j^{2}}+2\eta ,
\]
and $\beta_n(j)\to1$ for each fixed $j$, so the first summand tends to
$0$. 
Thus, the normalized sum in the preceding display tends to $\Psi_\nu(r)$ along each parity class.
\end{proof}

\begin{corollary}\label{cor:ratio}
Let $a_1b_1=a_2b_2=p$ and $r_i:=a_i/b_i$ for $i\in\{1,2\}$. 
Then, for each $\nu\in\{0,\frac12\}$,
$$
 \lim_{\substack{n\to\infty\\ \nu_n=\nu}}
 \frac{Z_{a_1,b_1}(n)}{Z_{a_2,b_2}(n)}
 =
 R_\nu:=\frac{\Psi_\nu(r_1)}{\Psi_\nu(r_2)}>0.
$$
\end{corollary}

\subsection{Separating the even and odd limits via Poisson summation}
\label{subsec:tool}
By Corollary~\ref{cor:ratio}, it remains to distinguish the positive
constants $R_0$ and $R_{1/2}$. 

A function $f\colon\mathbb R\to\mathbb C$ is a
\emph{Schwartz function} (see~\cite[Chapter~5, Section~1.3, p.~134f]{SteinShakarchi}) if it has continuous derivatives $f^{(i)}$ of every
order $i \in {\mathbb N}$ (where $f^{(0)}=f$) and, for all non-negative integers $p,q$,
\begin{align}
 \sup_{x\in\mathbb R}|x^p f^{(q)}(x)|<\infty. \label{eq:decay}
\end{align}
For $f\in\mathcal S(\mathbb R)$, its \emph{Fourier transform} is 
\[
 \widehat f(x)=\int_{\mathbb R}f(u)e^{-2\pi i x u}\,du
 \qquad(x \in\mathbb R).
\]
The set of these functions is denoted by $\mathcal S(\mathbb R)$.
The integral converges absolutely;
see~\cite[Chapter~5, Sections~1.1--1.2, pp.~131--134]{SteinShakarchi}.
\begin{theorem}[Poisson summation formula, {\cite[Chapter~5, Theorem~3.1, p.~154]{SteinShakarchi}}]\label{thm:poisson}
If $f\in\mathcal S(\mathbb R)$, then
\[
 \sum_{k\in {\mathbb Z}} f(x+k)
 =\sum_{m\in {\mathbb Z}}\widehat f(m)e^{2\pi imx}.
\]
\end{theorem}

We apply this theorem to the (Gaussian) function $x \mapsto e^{-\alpha x^2}$, which clearly satisfies~\eqref{eq:decay} for $\alpha > 0$ and $s$, and obtain the following identity.
\begin{corollary}[Poisson summation for a Gaussian]\label{cor:poisson-gauss}
For every $\alpha>0$ and every $s\in\mathbb R$,
\begin{equation}\label{eq:poisson-gauss}
  \sum_{k\in\mathbb Z} e^{-\alpha(k-s)^{2}}
  \;=\;
  \sqrt{\frac{\pi}{\alpha}}\sum_{m\in\mathbb Z}
  e^{-\pi^{2}m^{2}/\alpha}\,e^{-2\pi ims} .
\end{equation}
\end{corollary}
\begin{proof}
Let $f(x)=e^{-\alpha x^2}$. Since every derivative of $f$
is a polynomial times $f$, we have $f\in\mathcal S(\mathbb R)$.
The function $x\mapsto e^{-\pi x^2}$ is its own Fourier transform
\cite[Chapter~5, Theorem~1.4]{SteinShakarchi}.
To compute the Fourier transform of $f$, we substitute
$y=\sqrt{\alpha/\pi}\,u$ in its defining integral:
\begin{align*}
 \widehat f(x)
 &=\int_{\mathbb R}e^{-\alpha u^2}e^{-2\pi ixu}\,du =\sqrt{\frac{\pi}{\alpha}}
   \int_{\mathbb R}e^{-\pi y^2}
     e^{-2\pi i(\sqrt{\pi/\alpha}\,x)y}\,dy.
\end{align*}
The last integral is the Fourier transform of
$y\mapsto e^{-\pi y^2}$ evaluated at $\sqrt{\pi/\alpha}\,x$.
By the cited identity, it equals
$e^{-\pi(\sqrt{\pi/\alpha}\,x)^2}$. Hence,
\begin{align*}
 \widehat f(x)
 =\sqrt{\frac{\pi}{\alpha}}\,
   e^{-\pi(\sqrt{\pi/\alpha}\,x)^2} =\sqrt{\frac{\pi}{\alpha}}\,e^{-\pi^2x^2/\alpha}.
\end{align*}
Applying Theorem~\ref{thm:poisson} to $f$ at $x=-s$ gives
$\sum_{k\in\mathbb Z}f(k-s)
 =
 \sum_{m\in\mathbb Z}\widehat f(m)e^{-2\pi ims}$.
Substituting the expressions for $f$ and $\widehat f$, we obtain~\eqref{eq:poisson-gauss}, as required.
\end{proof}

Set $\rho=e^{-\pi^2/\log_e 2}$ and define 
\[ 
 \Xi(s) := \sum_{k\in\mathbb Z}2^{-(k-s)^2}=\sqrt{\pi/{\log_e 2}}\,
      \bigl(\sum_{m\in \mathbb Z}\rho^{\,m^{2}}e^{2\pi i ms}\bigr)  \text{ and }
 \psi(s):=\frac{\Xi(s+1/2)}{\Xi(s)}.
\]
It will be useful later to remark that $\rho$ is less than $10^{-6}$ (computed on Wolfram Alpha).
A function $f \colon {\mathbb R} \to {\mathbb R}$ is called \emph{$1$-periodic} if $f(s+1)=f(s)$ 
for every $s\in\mathbb R$,
\emph{even} if $f(-s)=f(s)$ for every $s\in\mathbb R$, and \emph{positive} if $f(s)>0$ for every $s\in\mathbb R$.

\begin{lemma}\label{lem:separation}
With $\Xi,\psi,\rho$ as above:
\begin{enumerate}
\item $\Xi(s)=\sqrt{\pi/{\log_e 2}}\,
      \bigl(1+2\sum_{m\ge1}\rho^{\,m^{2}}\cos(2\pi ms)\bigr)$;
      in particular, $\Xi$ is $1$-periodic, even, and positive.
\item $\Xi$ is strictly decreasing on $[0,\frac12]$, and $\psi$ is
      even, $1$-periodic and strictly increasing on $[0,\frac12]$.
\item For
$\nu\in\{0,\frac12\}$ and $r>0$ and $t:=\frac12\log_2 r$ we have
      $\Psi_\nu(r)=2^{\,t^{2}}\Xi(t-\nu)$; consequently, in the
      situation of Corollary~\ref{cor:ratio} and with
      $t_i:=\frac12\log_2 r_i$ for $i\in\{1,2\}$,
      \[
        \frac{R_{1/2}}{R_0}=\frac{\psi(t_1)}{\psi(t_2)},
        \qquad\text{so}\qquad
        R_0=R_{1/2}
        \iff
        \cos(2\pi t_1)=\cos(2\pi t_2).
      \]
\item Quantitatively,
\[
 \left|
 \frac{R_{1/2}}{R_0}
 -1+4\rho\bigl(\cos(2\pi t_1)-\cos(2\pi t_2)\bigr)
 \right|
 \le 50\rho^2.
\]
\end{enumerate}
\end{lemma}

\begin{proof}
(1) 
Since $\Xi(s)$ is real, pairing the terms $m$ and $-m$ replaces the two
exponentials by $2\cos(2\pi ms)$, which gives the stated formula.

(2) We first note the elementary inequality
\[
 |\sin(m\theta)|\le m|\sin\theta|
 \qquad(m\ge1,\ \theta\in\mathbb R).
\]
Indeed, the case $m=1$ is immediate, and the sine addition
formula gives
\[
 |\sin((m+1)\theta)|
 \le |\sin(m\theta)|+|\sin\theta|,
\]
which proves the inequality by induction.

The series in (1) may be differentiated term by term:
the series of derivatives converges uniformly because
$\sum_{m\ge1}m\rho^{m^2}<\infty$. Thus,
\[
 \Xi'(s)
 =-4\pi\sqrt{\frac{\pi}{\log_e2}}
 \left(
   \rho\sin(2\pi s)
   +\sum_{m\ge2}m\rho^{m^2}\sin(2\pi ms)
 \right).
\]
For $0<s<\frac12$, we have $\sin(2\pi s)>0$, so the
elementary inequality above implies
\[
 \left|
 \sum_{m\ge2}m\rho^{m^2}\sin(2\pi ms)
 \right|
 \le \sin(2\pi s)\sum_{m\ge2}m^2\rho^{m^2}.
\]
To bound the last sum, note that
$m^2-4\ge5(m-2)$ for every integer $m\ge2$.
Since $\rho<10^{-6}<1/2$, we obtain
\[
 \sum_{m\ge2}m^2\rho^{m^2}
 \le \rho^4\sum_{j\ge0}(j+2)^2\,2^{-5j}
 <5\rho^4<\rho.
\]
Consequently, the expression in parentheses in the formula
for $\Xi'(s)$ is at least
\[
 (\rho-5\rho^4)\sin(2\pi s)>0.
\]
Hence $\Xi'(s)<0$ for $0<s<\frac12$, proving that $\Xi$
is strictly decreasing on $[0,\frac12]$.
By the evenness and $1$-periodicity of $\Xi$, we have
\[
 \Xi(s+\tfrac12)=\Xi(\tfrac12-s).
\]
Thus, on $[0,\frac12]$, the numerator of
$\psi(s)=\Xi(s+\frac12)/\Xi(s)$ is positive and strictly
increasing, while its denominator is positive and strictly
decreasing. It follows that $\psi$ is strictly increasing
on this interval.
The $1$-periodicity of $\psi$ follows from that of $\Xi$. Finally, the evenness and periodicity of $\Xi$ give
\[
 \psi(-s)
 =\frac{\Xi(\frac12-s)}{\Xi(-s)}
 =\frac{\Xi(s-\frac12)}{\Xi(s)}
 =\frac{\Xi(s+\frac12)}{\Xi(s)}
 =\psi(s),
\]
so $\psi$ is even.

\ignore{ 
(2) The series may be differentiated term by term.  Moreover,
\[
  \sin(2\pi ms)=\sin(2\pi s)\,U_{m-1}(\cos 2\pi s),
\]
where $U_{m-1}$ is the \emph{Chebyshev polynomial of the second kind}\footnote{This is a classic of early analysis courses}, which
satisfies $|U_{m-1}|\le m$ on $[-1,1]$.
These polynomials are defined by $U_0(z)=1$, $U_1(z)=2z$,
and $U_{j+1}(z)=2zU_j(z)-U_{j-1}(z)$\todo[inline]{Do we actually need the definition of the polynomials? MB: I agree, we don't need them.} for $j\ge1$. Hence,
\[
  \Xi'(s)=-4\pi\sqrt{\tfrac{\pi}{{\log_e 2}}}\;\sin(2\pi s)
   \Bigl[\rho+\sum_{m\ge2}m\rho^{\,m^{2}}U_{m-1}(\cos2\pi s)\Bigr],
\]
and $\bigl|\sum_{m\ge2}m\rho^{m^{2}}U_{m-1}\bigr|
\le\sum_{m\ge2}m^{2}\rho^{\,m^{2}}\le5\rho^{4}<\rho$.
\begin{color}{red}
    For the last inequality, we can use  $\rho<10^{-6}$. For the middle inequality, we use that when $m\geq 2$ $\rho^{m^2-4}\leq  \rho^{5(m-2)}\leq2^{-5(m-2)}$, so
$\sum_{m\ge2}m^2\rho^{m^2-4}
\le\sum_{j\ge0}(j+2)^2 2^{-5j}<5$.
\end{color} So the bracket is
positive and $\Xi'(s)<0$ for $s\in(0,\frac12)$, proving that $\Xi$ is
strictly decreasing on $[0,\frac12]$.  By periodicity $\Xi(s+\frac12)$
is then strictly increasing on $[0,\frac12]$, so $\psi$ is a quotient of
a positive strictly increasing function by a positive strictly
decreasing one, hence strictly increasing there.  Finally, $\psi$ is
$1$-periodic, and
$\psi(-s)=\Xi(\frac12-s)/\Xi(-s)=\Xi(s-\frac12)/\Xi(s)=\psi(s)$ using
that $\Xi$ is even and $1$-periodic.}

(3) From $r^{\,j}2^{-j^{2}}=2^{\,t^{2}}2^{-(j-t)^{2}}$ and
$\{\,j-\nu \mid j\in\mathbb Z+\nu\,\}=\mathbb Z$ we get
$\Psi_\nu(r)=2^{\,t^{2}}\Xi(t-\nu)$.  The factors $2^{\,t_i^{2}}$
cancel in
\[
  \frac{R_{1/2}}{R_0}
  =\frac{\Psi_{1/2}(r_1)\Psi_0(r_2)}{\Psi_{1/2}(r_2)\Psi_0(r_1)}
  =\frac{\Xi(t_1-\frac12)\,\Xi(t_2)}{\Xi(t_2-\frac12)\,\Xi(t_1)}
  =\frac{\psi(t_1)}{\psi(t_2)},
\]
where the last step uses $\Xi(t-\frac12)=\Xi(t+\frac12)$.  By (2) the
function $\psi$ is even, $1$-periodic and strictly monotone on
$[0,\frac12]$, so $\psi(t_1)=\psi(t_2)$ if and only if
$t_1\equiv\pm t_2\pmod 1$, i.e., if and only if
$\cos2\pi t_1=\cos2\pi t_2$.

(4) Put $c_i:=\cos(2\pi t_i)$ for $i\in\{1,2\}$.
By (1), we may write
\[
 \Xi(t_i)
 =\sqrt{\frac{\pi}{\log_e2}}\,(1+2\rho c_i+e_i),
 \qquad
 \Xi(t_i-\tfrac12)
 =\sqrt{\frac{\pi}{\log_e2}}\,(1-2\rho c_i+e_i'),
\]
where
\[
 |e_i|,|e_i'|
 \le 2\sum_{m\ge2}\rho^{m^2}
 \le 3\rho^4.
\]
Consequently, the quotient in (3) is $R_{1/2}/R_0=N/D$, where
\begin{align*}
 N&:=(1-2\rho c_1+e_1')(1+2\rho c_2+e_2),\\
 D&:=(1+2\rho c_1+e_1)(1-2\rho c_2+e_2').
\end{align*}
Using $|c_i|\le1$ and $\rho<10^{-6}$, the total contribution
of the error terms in each product is at most
\[
 6\rho^4(1+2\rho)+9\rho^8\le7\rho^4.
\]
Without these error terms, the two products are
\[
 1-2\rho(c_1-c_2)-4\rho^2c_1c_2
 \quad\text{and}\quad
 1+2\rho(c_1-c_2)-4\rho^2c_1c_2.
\]
Their quadratic terms cancel when we subtract. Hence,
\[
 |N-D+4\rho(c_1-c_2)|\le14\rho^4,
 \qquad
 |D-1|\le4\rho+4\rho^2+7\rho^4\le5\rho.
\]
In particular, $D\ge1-5\rho>0$. Using $|c_1-c_2|\le2$,
we therefore obtain
\begin{align*}
 \left|\frac{N}{D}-1+4\rho(c_1-c_2)\right|
 &\le
 \frac{|N-D+4\rho(c_1-c_2)|
       +4\rho|c_1-c_2|\,|D-1|}{D}\\
 &\le \frac{14\rho^4+40\rho^2}{1-5\rho}
 \le 50\rho^2, 
\end{align*}
as claimed.
\end{proof}



We say that a pair of bipartite tournaments $C_1,C_2$ with
part sizes $a_i,b_i$ satisfies the \emph{parity-separation condition} if
\begin{equation}\label{eq:2adic}
  a_1b_1=a_2b_2=:P
  \qquad\text{and}\qquad
  \text{neither }\ \frac{a_1}{a_2}\ \text{ nor }\ \frac{a_1a_2}{P}
  \ \text{ is an integral power of }2 .
\end{equation}

\begin{lemma}\label{lem:arith}
Let $r_1,r_2>0$ and for $i\in\{1,2\}$ let $t_i:=\frac12\log_2r_i$.  Then
$\cos2\pi t_1=\cos2\pi t_2$ if and only if $r_1/r_2 \in4^{\mathbb Z}$ or $r_1r_2 \in4^{\mathbb Z}$.  If moreover $r_i=a_i/b_i$ with
$a_1b_1=a_2b_2=p$, this happens if and only if $a_1/a_2 \in2^{\mathbb Z}$ or $a_1a_2/p \in2^{\mathbb Z}$.
\end{lemma}

\begin{proof}
$\cos2\pi t_1=\cos2\pi t_2$ iff $t_1-t_2\in\mathbb Z$ or
$t_1+t_2\in\mathbb Z$.  Now $t_1-t_2=\frac12\log_2(r_1/r_2)$ lies in
$\mathbb Z$ iff $r_1/r_2\in4^{\mathbb Z}$, and likewise
$t_1+t_2=\frac12\log_2(r_1r_2)$ iff $r_1r_2\in4^{\mathbb Z}$.  If $a_ib_i=p$ then
$r_1/r_2= a_1 b_2 / a_2 b_1 = (a_1/a_2)^{2}$ and $r_1r_2=a_1a_2/b_1b_2 = (a_1a_2/p)^{2}$, and a positive
rational $x$ satisfies $x^{2}\in4^{\mathbb Z}$ iff $x\in2^{\mathbb Z}$.
\end{proof}

By Corollary~\ref{cor:ratio} and Lemmas~\ref{lem:separation}
and~\ref{lem:arith}, condition \eqref{eq:2adic} says precisely that
$Z_{a_1,b_1}(n)/Z_{a_2,b_2}(n)$ converges along the even and along the
odd integers to two \emph{distinct} positive limits $R_0\ne R_{1/2}$.
Note that \eqref{eq:2adic} is invariant under exchanging $a_i$ and
$b_i$, and that it forces $\{a_1,b_1\}\ne\{a_2,b_2\}$: if $a_1=a_2$ then
$a_1/a_2=2^{0}$, and if $a_1=b_2$ and $b_1=a_2$ then
$a_1a_2/P=1=2^{0}$.


\subsection{Non-convergence for \texorpdfstring{$D$}{D}-coloured digraphs}
The rest of this section will be dedicated to applying the above results in order to prove non-convergence of the proportion of structures satifying certain sentences in given classes. All the classes will be constists of digraphs mapping to a bipartite digraph, i.e., $\Col(D)$ or $\Csp(D)$.
In both cases $D$ will be the union of two non-isomorphic bipartite digraph with the same density and entropy. In the case of $\Col(D)$ any two such digraphs satisfying the parity-separation condition condition on the number of vertices in each part will suffice. In the case of $\Csp(D)$, we also need that the two graphs are twin-free (in particular, not homomorphically equivalent). 

We first treat $\Col(D)$, i.e., the class of $D$-coloured digraphs
regarded as relational structures $G_f$ in the signature
$\{E\}\cup\{P_u \mid u\in V(D)\}$ as in Definition~\ref{def:D-col}.  

\begin{proposition}\label{prop:strata}
For $i=1,2$ let $C_i$ be a bipartite tournament with parts of sizes
$a_i,b_i$ satisfying $a_1b_1=a_2b_2=P$, put $D:=C_1\uplus C_2$,
$Z_i(n):=Z_{a_i,b_i}(n)$ and
$\mathcal E^n_i:=\{(G,f)\in\Col(D)_n \mid \mathrm{im} f=V(C_i)\}$.  Then
\[
  |{\Col}(D)_n|=(1+o(1))\bigl(|\mathcal E^n_1|+|\mathcal E^n_2|\bigr)
  \qquad\text{and}\qquad
  |\mathcal E^n_i|=(1+o(1))\,Z_i(n).
\]
\end{proposition}

\begin{proof}
We have $S(D)=\{C_1,C_2\}$ (see Example~\ref{expl:S(D)}).
Both assertions now follow directly from
Corollary~\ref{cor:optimal-weights} and 
Lemma~\ref{lem:Z-exact}.
In particular, surjectivity follows from the concentration theorem
even when one of the parts is a singleton.
\end{proof}

\begin{theorem}\label{thm:coloured}
Let $C_1,C_2$ be bipartite tournaments satisfying \eqref{eq:2adic} and
put $D:=C_1\uplus C_2$.  Then $\Col(D)$ does not have first-order
convergence.  Indeed, for any $u\in V(C_1)$ the sentence
$\varphi:=\exists x\,P_u(x)$ satisfies
\[
  \varphi^{\Col(D)}_n\longrightarrow\frac{R_0}{1+R_0}
  \quad\text{along even }n,
  \qquad
  \varphi^{\Col(D)}_n\longrightarrow\frac{R_{1/2}}{1+R_{1/2}}
  \quad\text{along odd }n,
\]
and these two limits are distinct.
\end{theorem}

\begin{proof}
Every $(G,f)\in\mathcal E^n_1$ satisfies $\varphi$, because $u$ lies in
the image of $f$, and every $(G,f)\in\mathcal E^n_2$ fails $\varphi$,
because $u\notin V(C_2)$.  Hence, by Proposition~\ref{prop:strata},
\[
  \varphi^{\Col(D)}_n
  =\frac{|\mathcal E^n_1|}{|\mathcal E^n_1|+|\mathcal E^n_2|}+o(1)
  =\frac{Z_1(n)/Z_2(n)}{1+Z_1(n)/Z_2(n)}+o(1),
\]
and $Z_1(n)/Z_2(n)\to R_{\nu}$ along the integers with $\nu_n=\nu$ by
Corollary~\ref{cor:ratio}.  Since $x\mapsto x/(1+x)$ is injective on
$(0,\infty)$ and $R_0\ne R_{1/2}$ by \eqref{eq:2adic}, the two limits
differ.
\end{proof}

\begin{corollary}\label{cor:coloured-example}
Let $D:=\vec K_{1,9}\uplus\vec K_{3,3}$, an oriented graph with
$10+6=16$ vertices.  Then $\Col(D)$ does not have first-order
convergence.
\end{corollary}

\begin{proof}
Here $P=1\cdot9=3\cdot3=9$, and $a_1/a_2=\frac13$ and
$a_1a_2/P=\frac39=\frac13$ are not powers of $2$, so \eqref{eq:2adic}
holds.  Apply Theorem~\ref{thm:coloured}.
\end{proof}

\begin{remark}\label{rem:coloured-numerics}
For Corollary~\ref{cor:coloured-example} we have $r_1=\frac19$ and
$r_2=1$, hence $t_1=-\log_23$ and $t_2=0$, so that
$\cos2\pi t_1=\cos(2\pi\log_23)=-0.86086\ldots$ and $\cos2\pi t_2=1$;
these differ simply because $\log_23\notin\mathbb Z$.  Numerically
$R_0=5.7045085\ldots$ and $R_{1/2}=5.7045363\ldots$, so the two
subsequential limits of $\varphi^{\Col(D)}_n$ are $0.85084663\ldots$ and
$0.85084725\ldots$.  Moreover, $16$ vertices is optimal for this
construction: since $a_i+b_i\ge2\sqrt P$, a pair satisfying
\eqref{eq:2adic} with fewer than $16$ vertices in total would need
$P\le15$.  An inspection of the finitely many cases shows that among
these only $P=9$, with $\{1\cdot9,\,3\cdot3\}$, and $P=15$, with
$\{1\cdot15,\,3\cdot5\}$ and hence $24$ vertices, admit such a pair at
all. For instance for $P=8$ the only two factorizations are $1\cdot8$
and $2\cdot4$, and there $a_1/a_2=\frac12$ is a power of $2$.
\end{remark}


\subsection{Removing the colouring}

Theorem~\ref{thm:coloured} does not yield a class of the form
$\CSP(D)$: in
Corollary~\ref{cor:coloured-example} both components of $D$ are
homomorphically equivalent to a single arc, so
$\CSP(D)=\CSP(\vec K_{1,1})$, and the almost-sure theory of $\CSP(D)$
is $\omega$-categorical by Proposition~\ref{prop:no-core}.  To pass from $\Col(D)$ to
$\CSP(D)$ we need, with probability tending to one, any two
$D$-colourings of the sampled digraph to differ only by
postcomposition with an automorphism of $D$.
This is the role of the following two lemmas.

\begin{lemma}\label{lem:bt-hom}
Let $C$ and $C'$ be bipartite tournaments with parts $(L,R)$ and
$(L',R')$, all four parts non-empty, and let $h\colon C\to C'$ be a
homomorphism.  Then
\begin{enumerate}
\item $h(L)\subseteq L'$ and $h(R)\subseteq R'$, or $h(L)\subseteq R'$
      and $h(R)\subseteq L'$;
\item if $C$ is twin-free, then $h$ is injective.
\end{enumerate}
\end{lemma}

\begin{proof}
(1) Suppose towards a contradiction that $u,v\in L$ satisfy $h(u)\in L'$
and $h(v)\in R'$, and pick $w\in R$.  Then $w$ is adjacent to both $u$
and $v$, so $h(w)$ is adjacent to both $h(u)$ and $h(v)$.  Adjacency to
$h(u)\in L'$ forces $h(w)\in R'$ and adjacency to $h(v)\in R'$ forces
$h(w)\in L'$, a contradiction.  The same argument applies with $L$ and
$R$ interchanged.
their images must lie in opposite parts of $C'$. This proves (1).

(2) By (1), vertices in different parts have images in different
parts. Thus, if $h$ is not injective, it identifies two distinct
vertices $u,v$ in the same part, say in $L$.
By twin-freeness, there is $w\in R$ with, say, $u\to w$ and $w\to v$.  Then
$h(u)\to h(w)$ and $h(w)\to h(v)=h(u)$, which is impossible.
\end{proof}

In particular a bipartite tournament is a core if and only if it is
twin-free; this is the case $\ell=2$ of Lemma~\ref{lem:twins}.

\begin{lemma}\label{lem:incomparable}
For $i=1,2$ let $C_i$ be a twin-free bipartite tournament whose parts
have sizes $a_i,b_i\ge1$, and suppose that $a_1b_1=a_2b_2$ and
$\{a_1,b_1\}\ne\{a_2,b_2\}$.  Then $C_1\not\to C_2$ and
$C_2\not\to C_1$, and $D:=C_1\uplus C_2$ is a core.
\end{lemma}

\begin{proof}
By Lemma~\ref{lem:bt-hom} a homomorphism $C_1\to C_2$ is injective and
maps parts into parts, so after possibly exchanging $a_2$ and $b_2$ we
have $a_1\le a_2$ and $b_1\le b_2$.  Then $a_1b_1\le a_2b_2=a_1b_1$
forces $a_1=a_2$ and $b_1=b_2$, contradicting
$\{a_1,b_1\}\ne\{a_2,b_2\}$; symmetrically in the other direction.

Now let $g$ be an endomorphism of $D$.  As $C_i$ is connected, $g(C_i)$
is contained in a single component $C_j$ of $D$, and $j=i$ by the above.
So $g\restriction C_i$ is an endomorphism of the core $C_i$, hence an
automorphism, and therefore $g\in\Aut(D)$.
\end{proof}

Recall the first-order $\{E\}$-sentence $\sigma_H$ attached to a finite
core digraph $H$ in Lemma~\ref{lem:recovery}.  In the following, $C_1$ and $C_2$ are
twin-free bipartite tournaments with $a_1b_1=a_2b_2=P$ and
$\{a_1,b_1\}\ne\{a_2,b_2\}$, and $D:=C_1\uplus C_2$,
$Z_i(n):=Z_{a_i,b_i}(n)$, $A_i:=|\Aut(C_i)|$ for $i\in\{1,2\}$.

\begin{lemma}\label{lem:recover}
Let $i\in\{1,2\}$, let $G\in\CSP(C_i)_n$ satisfy $G\models\sigma_{C_i}$,
and let $f\in\Hom(G,C_i)$.  Then
\[
  \Hom(G,D)=\Hom(G,C_i)
           =\{\alpha\circ f \mid \alpha\in\Aut(C_i)\},
\]
so $|\Hom(G,D)|=A_i$; moreover $G\notin\CSP(C_{3-i})$.
\end{lemma}

\begin{proof}
The second equality is Lemma~\ref{lem:recovery} applied to the core $C_i$.  For the
first, let $h\in\Hom(G,D)$ and $x\in V(G)$.  The sentence
$\sigma_{C_i}$ supplies a tuple $\bar b$ inducing a copy of $C_i$ in $G$
together with the arcs between $x$ and $\bar b$ prescribed by
$\eta_u(\bar a,x,\bar b)$ for some $u\in V(C_i)$; since $u$ has a
neighbour in $C_i$, the substructure of $G$ induced by
$\{x\}\cup\bar b$ is connected and contains a copy of $C_i$.  Its
$h$-image therefore lies in a single component $C_j$ of $D$ and admits a
homomorphism from $C_i$, so $j=i$ by Lemma~\ref{lem:incomparable}.
Hence, $h(x)\in V(C_i)$ for every $x$, i.e., $h\in\Hom(G,C_i)$.  Finally,
$G$ contains a copy of $C_i$ and $C_i\not\to C_{3-i}$, so
$G\not\to C_{3-i}$.
\end{proof}

\begin{proposition}\label{prop:uncolour}
Let $\varphi$ be the conjunctive query of $C_1$, i.e., the primitive
positive sentence
$\exists(x_v)_{v\in V(C_1)}\bigwedge_{(u,v)\in E(C_1)}E(x_u,x_v)$, so
that $G\models\varphi$ if and only if $C_1\to G$.  Then
\[
  |\CSP(D)_n|=(1+o(1))\Bigl(\frac{Z_1(n)}{A_1}+\frac{Z_2(n)}{A_2}\Bigr)
  \qquad\text{and}\qquad
  \varphi^{\CSP(D)}_n
  =\frac{Z_1(n)/A_1}{\dfrac{Z_1(n)}{A_1}+\dfrac{Z_2(n)}{A_2}}+o(1).
\]
\end{proposition}

\begin{proof}
Put $\mathcal G_n^i=\{G\in\Csp(C_i)_n \mid G\models\sigma_{C_i}\}$.
These sets are disjoint by Lemma~\ref{lem:recover}, and every digraph
in $\mathcal G_n^i$ has exactly $A_i$ colourings into $D$, all
surjective onto $C_i$. Proposition~\ref{prop:Kl} and
Lemma~\ref{lem:recovery} imply
\[
 A_i|\mathcal G_n^i|=(1+o(1))Z_i(n).
\]
Let $\mathcal B_n$ be the complement of their union in $\Csp(D)_n$.
Counting fibres gives
\begin{equation}\label{eq:fibres}
 |{\Col}(D)_n|=A_1|\mathcal G_n^1|+A_2|\mathcal G_n^2|
             +\sum_{G\in\mathcal B_n}|\Hom(G,D)|.
\end{equation}
Proposition~\ref{prop:strata} implies that the last sum is
$o(Z_1(n)+Z_2(n))$. Since every such digraph has at least one
colouring, $|\mathcal B_n|$ satisfies the same bound. Thus,
\[
 |\Csp(D)_n|=(1+o(1))\left(Z_1(n)/A_1+Z_2(n)/A_2\right).
\]
Every digraph in $\mathcal G_n^1$ contains a copy of $C_1$ and
satisfies its conjunctive query $\varphi$. No digraph in
$\mathcal G_n^2$ satisfies it, since $C_1\not\to C_2$.
The bound on $\mathcal B_n$ now gives the probability formula.
\end{proof}

\subsection{Non-convergence for \texorpdfstring{$\CSP(D)$}{CSP(D)}}
We now combine the parity-dependent counting formulas with
Proposition~\ref{prop:uncolour} to obtain non-convergence after
forgetting the colouring. For twin-free components, passing to
uncoloured counts divides each component's coloured count by the
size of its automorphism group. These fixed factors do not remove
the difference between the even and odd limiting ratios.
The conjunctive query of $C_1$ distinguishes the two components
and therefore has different limiting probabilities along even and
odd orders. We first state this conclusion for any pair satisfying
our hypotheses, and then give explicit matrices producing the
$25$-vertex core promised in Theorem~\ref{thm:non-convergence}.

\begin{theorem}\label{thm:nonconv}
For $i=1,2$ let $C_i$ be a \emph{twin-free} bipartite tournament with
parts of sizes $a_i,b_i$, and suppose that $C_1,C_2$ satisfy
\eqref{eq:2adic}.  Then $D:=C_1\uplus C_2$ is a core oriented graph and
$\CSP(D)$ does not have first-order convergence.  More precisely, with
$\kappa:=|\Aut(C_2)|/|\Aut(C_1)|$, the conjunctive query $\varphi$ of
$C_1$ satisfies
\[
  \varphi^{\CSP(D)}_n\longrightarrow\frac{\kappa R_0}{1+\kappa R_0}
  \quad\text{along even }n,
  \qquad
  \varphi^{\CSP(D)}_n\longrightarrow
  \frac{\kappa R_{1/2}}{1+\kappa R_{1/2}}
  \quad\text{along odd }n,
\]
and these two limits are distinct.
\end{theorem}

\begin{proof}
As noted after \eqref{eq:2adic} we have $\{a_1,b_1\}\ne\{a_2,b_2\}$, so
$D$ is a core by Lemma~\ref{lem:incomparable}, and
$S(D)=\{C_1,C_2\}$ by Lemma~\ref{expl:S(D)}.  By
Proposition~\ref{prop:uncolour},
\[
  \varphi^{\CSP(D)}_n
  =\frac{\kappa\,Z_1(n)/Z_2(n)}{1+\kappa\,Z_1(n)/Z_2(n)}+o(1),
\]
and $Z_1(n)/Z_2(n)\to R_\nu$ along the integers with $\nu_n=\nu$ by
Corollary~\ref{cor:ratio}.  Since $x\mapsto\kappa x/(1+\kappa x)$ is
injective on $(0,\infty)$ and $R_0\ne R_{1/2}$ by \eqref{eq:2adic}, the
two limits differ.
\end{proof}

\begin{corollary}[Theorem~\ref{thm:non-convergence}]\label{cor:example}
Let
\[
  M_1=\begin{pmatrix}
   0&0&0&0&0&0&0&0&1\\
   0&0&0&0&1&1&1&1&0\\
   0&0&1&1&0&0&1&1&0\\
   0&1&0&1&0&1&0&1&0
  \end{pmatrix},
  \qquad
  M_2=I_6 ,
\]
where $I_6$ is the $6\times6$ identity matrix, and
put $C_1:=T(M_1)$, $C_2:=T(M_2)$ and $D:=C_1\uplus C_2$.  Then $D$ is a
core oriented graph with $13+12=25$ vertices, and $\CSP(D)$ does not
have first-order convergence.
\end{corollary}

\begin{proof}
The columns of $M_1$ are the binary expansions of $0,\dots,8$, hence
pairwise distinct, and its rows are pairwise distinct as well; $I_6$ has
pairwise distinct rows and columns.  So $C_1$ and $C_2$ are twin-free.
Here $P=4\cdot9=6\cdot6=36$, and $a_1/a_2=\frac46=\frac23$ and
$a_1a_2/P=\frac{24}{36}=\frac23$ are not powers of $2$, so
\eqref{eq:2adic} holds.  Now apply Theorem~\ref{thm:nonconv}.
\end{proof}

\section{Concluding Remarks and Open Problems}
Theorem~\ref{thm:main} shows that, for a finite non-empty oriented
template $D$, a uniformly random $(G,f)\in\Col(D)_n$
asymptotically almost surely has $\operatorname{im}f=V(Q)$ for some
$Q\in S(D)$, where $Q$ is an orientation of a complete multipartite
graph. For arbitrary digraphs the density is still
given by a quadratic form, but the Motzkin--Straus description no
longer applies directly because opposite arcs are allowed.
For relations of higher arity the analogous counting expressions
involve higher-degree polynomials. Extending the structural part
of the argument to these settings remains an interesting problem.

\begin{question}
Is there a finite oriented graph $D$ such that $\Csp(D)$ fails
first-order convergence along every infinite arithmetic progression?
\end{question}

\begin{question}
Does $\Forb(\{\vec C_3\})$ have first-order convergence?
Determine its almost-sure theory. Can its asymptotic first-order
behaviour be described by finitely many $\omega$-categorical theories?
\end{question}


\subsection*{Acknowledgements}
The authors thank an anonymous referee of an earlier version who pointed out an incorrect argument in a central proof of an earlier version of the article. We also want to thank Santiago Guzm\'an Pro for many conversations during the process of fixing this argument.

\subsection*{AI assistance}
The first version of this paper was put on ArXiv in 2022. In 2026, ChatGPT 5.6 Sol found a mistake in the main argument and a counterexample in what used to be the main theorem. We have since revised the work to focus on what our old methods where able to prove and integrated the AI's counterexample (Section 5). In particular, the AI was able to find the work of Read and Wright as relevant to our context which lead to the results as they are.

We used Claude 5 Opus and ChatGPT 6 Astra for reviewing and editing. The authors revised the entire document and take full responsibility of the resulting text. 


\end{document}